\documentclass[12pt,final]{amsart}

\usepackage[utf8]{inputenc}
\usepackage[english]{babel}
\usepackage{cite}
\usepackage{amsthm,amsmath,amssymb,hyperref,amsfonts,adjustbox}
\numberwithin{equation}{section}
\usepackage[shortlabels]{enumitem}
\usepackage[a4paper, total={6in, 8in}]{geometry}
\usepackage{imakeidx}

\usepackage{tikz}
\usetikzlibrary{cd}
\usetikzlibrary{calc} 
\usetikzlibrary{decorations.pathmorphing} 

\tikzset{curve/.style={settings={#1},to path={(\tikztostart)
    .. controls ($(\tikztostart)!\pv{pos}!(\tikztotarget)!\pv{height}!270:(\tikztotarget)$)
    and ($(\tikztostart)!1-\pv{pos}!(\tikztotarget)!\pv{height}!270:(\tikztotarget)$)
    .. (\tikztotarget)\tikztonodes}},
    settings/.code={\tikzset{quiver/.cd,#1}
        \def\pv##1{\pgfkeysvalueof{/tikz/quiver/##1}}},
    quiver/.cd,pos/.initial=0.35,height/.initial=0}

\tikzset{tail reversed/.code={\pgfsetarrowsstart{tikzcd to}}}
\tikzset{2tail/.code={\pgfsetarrowsstart{Implies[reversed]}}}
\tikzset{2tail reversed/.code={\pgfsetarrowsstart{Implies}}}
\tikzset{no body/.style={/tikz/dash pattern=on 0 off 1mm}}

\newtheorem{theorem}{Theorem}[section]
\newtheorem{corollary}[theorem]{Corollary}
\newtheorem{lemma}[theorem]{Lemma}
\newtheorem{proposition}[theorem]{Proposition}

\theoremstyle{definition}
\newtheorem{definition}[theorem]{Definition}
\newtheorem{example}[theorem]{Example}

\theoremstyle{remark}
\newtheorem{remark}[theorem]{Remark}

\newcommand*{\im}{\operatorname{im}}
\newcommand*{\dom}{\operatorname{dom}}

\title[The category of partial actions]{The category of partial group actions: quotients, (co)limits and groupoids}



\author{Emmanuel Jerez}
\address[Emmanuel Jerez]{Departamento de Matem{\'a}tica, Universidade de S\~ao Paulo,
Rua do Mat\~ao, 1010, 05508-090 S\~ao Paulo, Brazil}
\email{ejerez@ime.usp.br}

\begin{document}

 \maketitle


\begin{abstract}
    We consider the category of partial actions, where the group and the set upon which the group acts can vary. Within this framework, we develop a theory of quotient partial actions and prove that this category is both (co)complete and encompasses the category of groupoids as a full subcategory. 
    In particular, we establish the existence of a pair of adjoint functors, denoted as $\Phi : \textbf{Grpd} \to \textbf{PA}$ and $\Psi : \textbf{PA} \to \textbf{Grpd}$, with the property that $\Psi \Phi \cong 1_{\textbf{Grpd}}$. 
    Next, for a given groupoid $\Gamma$, we provide a characterization of all partial actions that allow the recovery of the groupoid $\Gamma$ through $\Psi$. This characterization is expressed in terms of certain normal subgroups of a universal group constructed from $\Gamma.$
\end{abstract}

\section*{Introduction}

Let $G$ be a group and $X$ be a set. A partial group action is a family of bijections $\theta_g: X_{g^{-1}} \to X_g$ between subsets of $X$, such that $X_1 = X$ and $\theta_g \circ \theta_h$ is a restriction of $\theta_{gh}$ for all $g, h \in G$. Note that in this case, $\circ$ is not the usual composition but the \textit{partial} composition (see \cite[Definition 2.1]{E6}). Set-theoretical partial group actions were introduced by R. Exel in \cite{exe} as a natural generalization of group actions, strongly connected with the theory of inverse semigroups. Partial actions of groups is a rich area of research. Therefore, studying their properties as a category is a natural step towards understanding this kind of structure and its relations with other categories, such as the category of groupoids. These categories have shown strong connections with each other, as evidenced by the various applications of groupoids in the theory of partial group actions, as we can see in works such as \cite{FAPAAG}, \cite{BVDG}, \cite{BVGD2}, \cite{DE2}, \cite{EGC} and \cite{DEP}. 

Groupoids and partial actions of groups share many similar definitions and constructions. For example, the concept of a groupoid action is similar to the idea of a partial action, and representations of groupoids and partial group representations exhibit similarities. Furthermore, both structures have the potential to be tools for studying local phenomena that global actions of groups cannot fully capture. It is well-known that any partial group action determines a natural groupoid called the partial action groupoid, and that any connected groupoid can be obtained from a (not necessarily unique) global group action. All these facts suggest the existence of strong connections between both concepts. This work focuses on the study of the category of partial group actions with the aim of understanding such a relationship.

We begin by defining  the category of partial group actions \textbf{PA} such that both the group and the set may vary, this differs from the usual definition where one assumes that the group is fixed (see \cite{E6}). 
Next, we introduce the concepts of partial group action congruence and quotient partial group action. While previous studies have examined quotient partial group actions (see \cite{AVPACAT}, \cite{DiBaHePi},\cite{KuoSG}), these papers often focused on specific relations or imposed significant restrictions. In contrast, our paper offers a general approach to the quotient of partial group actions, resulting in a proper construction of quotients within the category of partial group actions. Specifically, for any congruence on a partial group action $\theta$, we define a corresponding partial group action $\overline{\theta}$ along with a natural surjective morphism $\pi: \theta \to \overline{\theta}$ that satisfies the usual properties of quotient spaces. Conversely, any morphism of partial group actions determines a unique congruence of partial group actions. Subsequently, we prove that the category of partial group actions possesses small (co)limits and (co)equalizers. It is worth noting that this problem has been previously explored in \cite{AVPACAT}, albeit in the context of partial group actions with a fixed group and strong morphisms of partial group actions.

Subsequently, we explore the relationship between the category of partial group actions \textbf{PA} and the category of groupoids \textbf{Grpd}. Thanks to the existence of coproducts in the category of partial group actions we observe that any groupoid is a partial action groupoid, furthermore we observe that global actions of groupoids correspond to partial group actions. Motivated by these observations, we establish the existence of a pair of functors $\Phi : \textbf{Grpd} \to \textbf{PA}$ and $\Psi : \textbf{PA} \to \textbf{Grpd}$, with the property that $\Psi \Phi \cong 1_{\textbf{Grpd}}$. 

Finally, we focus on characterizing partial group actions that share the same associated partial action groupoid $\Gamma$. Specifically, we prove that these partial actions are exclusively determined by the normal subgroups of a group constructed from the groupoid $\Gamma$. For more detailed information on partial actions, please refer to the book \cite{E6} and the surveys \cite{EBSurvery} and \cite{D3}.


 \section{Partial group actions}

\subsection{The category of partial group actions}
We follow the book \cite{exe} for the basic theory of partial group actions.

 \begin{definition}
     A partial action $\theta = \big( G, X, \{ X_{g}\}_{g \in G}, \{ \theta_{g} \}_{g \in G} \big)$ of a group $G$ on a set $X$ consists of a family indexed by $G$ of subsets $X_{g} \subseteq X$ and a family of bijections $\theta_{g}: X_{g^{-1}} \to X_{g}$ for each $g \in G$, satisfying the following conditions:
     \begin{enumerate}[(i)]
        \item $X_{1} = X$ and $\theta_{1} = 1_{X}$,
        \item $\theta_{g}(X_{g^{-1}}\cap X_{g^{-1}h}) \subseteq X_{g} \cap X_{h}$,
        \item $\theta_{g}\theta_{h}(x) = \theta_{gh}(x)$, for each $x \in X_{h^{-1}} \cap X_{h^{-1}g^{-1}}$.
     \end{enumerate}
 \end{definition}

 It is worth to mention that the conditions $(ii)$, $(iii)$ are equivalent to say that $\theta_{g} \theta_{h}$ is a restriction of $\theta_{gh}$, where the composition $\theta_{g}\theta_{h}$ is the \textit{partial composition}, i.e., the domain of $\theta_{g}\theta_{h}$ is $\{ x \in X_{h^{-1}} : \theta_{h}(x) \in X_{g^{-1}}\}$.

 \begin{definition} \label{d: morphism of partial actions}
   Let $\alpha = (G, X, \{ X_g \}, \{ \alpha_g \}) $ and $\beta = (S, Y, \{ Y_h \},\{ \beta _{h}\}) $ be partial actions. A morphism of partial actions $\phi$ is a pair $(\phi_0, \phi_1)$ where $\phi_0: X \to Y$ is a map of sets and $\phi_1: G \to S$ is a morphism of groups such that
   \begin{enumerate}[(i)]
   	\item $ \phi_0(X_g) \subseteq  Y_{\phi_1(g)}$,
	\item $ \phi_0(\alpha_g(x)) = \beta_{\phi_1(g)}(\phi_0(x)) $, for all $x \in X_{g^{-1}}$.
   \end{enumerate}
 \end{definition}

 Note that Definition~\ref{d: morphism of partial actions} is a generalization of the concept of a $G$-equivariant map between partial group actions (see \cite[Definition 2.7]{exe}). Indeed, if in Definition~\ref{d: morphism of partial actions} we set $G=S$ and $\phi_1 = 1_G$, then we obtain a $G$-equivariant map.

 \begin{remark} \label{r: equality of partial actions}
    It is clear that if $\phi: \alpha \to \beta$ and $\varphi: \alpha \to \beta$ morphisms of partial actions, then $\phi = \varphi$ if, and only if, $\phi_0 = \varphi_0$ and $\phi_1 = \varphi_1.$
 \end{remark} 

 The category \textbf{PA} is the category whose objects are partial group actions and whose morphisms are morphisms of partial group actions.

 \subsection{Quotient partial actions}

 Give a proper definition about what means a congruence on a partial action is a task undone until now. The conventional definition of congruence for a global group action, which states that if $x$ and $y$ are related ($x \sim y$), then their images under any group element $g$ must also be related ($g \cdot x \sim g \cdot y$), does not seamlessly apply to the case of partial actions. This situation often imposes stricter conditions on the relation, such as requiring that if an element $x$ belongs to the domain of the partial action of a group element $g$ (i.e., $x \in X_g$), then the equivalence class of $x$ must be entirely contained within $X_g.$ However, these stricter conditions may lead to a loss of valuable information about partial actions and their morphisms. In this section, we will define a set of relations that every morphism of partial actions must satisfy. These relations are sufficient to provide a well-defined construction of a quotient partial action in a general sense. In this section $\alpha$ will denote a partial action $\alpha = (G, X, \{ X_g \}, \{ \alpha_g \})$ of $G$ on $X$.
 
 \begin{definition}
    Let $\alpha = (G, X, \{ X_g \}, \{ \alpha_g \})$ be a partial action, and let $\mathcal{R}$ be an equivalence relation on $X$. We say that an $n$-tuple
    \[
        (x_1 \delta_{g_1}, x_2 \delta_{g_2}, \ldots, x_n \delta_{g_n})
    \]
     is an $\mathcal{R}$-chain of length $n$ (the $\delta_g's$ acts as place-holders), if it satisfies the following conditions:
     \begin{enumerate}[(i)]
        \item $x_i \in X_{g_i}$ for all $i \in \{ 1, \ldots, n \}$; 
        \item if $n \geq 2$, $\alpha_{g_{i-1}^{-1}}(x_{i-1}) \sim_{\mathcal{R}} x_i$, for all $1 < i \leq n$.
     \end{enumerate}
 \end{definition}
 We can visualize an $\mathcal{R}$-chain $(x_1 \delta_{g_1}, \ldots, x_n \delta_{g_n})$ of length $n$ as the following diagram
 \[
     \begin{tikzcd}[row sep=tiny]
     {x_1} & {\alpha_{g_1^{-1}}(x_1)} \\
     & {x_2} & {\alpha_{g_2^{-1}}(x_2)} \\
     && {x_3} & {\alpha_{g_3^{-1}}(x_3)} \\
     &&& {x_n} & {\alpha_{g_n^{-1}}(x_n)},
     \arrow["\sim"{marking}, draw=none, from=2-2, to=1-2]
     \arrow["{g_3}"', maps to, from=3-4, to=3-3]
     \arrow["\sim"{marking}, draw=none, from=3-3, to=2-3]
     \arrow["{g_2}"', maps to, from=2-3, to=2-2]
     \arrow["{g_n}"', maps to, from=4-5, to=4-4]
     \arrow["\cdots"{marking}, draw=none, from=4-4, to=3-4]
     \arrow["{g_1}"', maps to, from=1-2, to=1-1]
 \end{tikzcd}
 \]
where the vertical symbol $\sim$ means $\mathcal{R}$-equivalent. 

\begin{definition}
    We define $\mathfrak{C}_{\alpha, \mathcal{R}}$ (or just $\mathfrak{C}$ if there is no ambiguity) as the set of all $\mathcal{R}$-chains of a partial action $\alpha$.
\end{definition}

    Let $\mathfrak{c}=(x_1 \delta_{g_1}, \ldots, x_n \delta_{g_n})$ and $\mathfrak{o}=(y_1 \delta_{h_1}, \ldots, y_m \delta_{h_m})$ be two $\mathcal{R}$-chains. Then, if $\alpha_{g_n^{-1}}(x_n) \sim y_1$ we can concatenate both $\mathcal{R}$-chains into a new $\mathcal{R}$-chain
    \[
        \mathfrak{c} \cdot \mathfrak{o} := (x_1 \delta_{g_1}, \ldots, x_n \delta_{g_n}, y_1 \delta_{h_1}, \ldots, y_m \delta_{h_m}). 
    \]
    Now, if $\mathfrak{c}=(x_1 \delta_{g_1}, \ldots, x_n \delta_{g_n})$ is an $\mathcal{R}$-chain we define  	
    \begin{equation*}
        \operatorname{t}(\mathfrak{c}):= x_1 \text{ and } \operatorname{s}(\mathfrak{c}) := \alpha_{g_n^{-1}}(x_n).
    \end{equation*}
    Thus, for any pair of $\mathcal{R}$-chains $\mathfrak{c}$ and $\mathfrak{o}$, the concatenation $\mathcal{R}$-chain $\mathfrak{c} \cdot \mathfrak{o}$ there exists if, and only if, $\operatorname{s}(\mathfrak{c}) \sim \operatorname{t}(\mathfrak{o})$. Moreover, any $\mathcal{R}$-chain have an ``\textit{inverse}'', as we see in the following lemma:

\begin{lemma}
    Let $\mathfrak{c}=(x_1 \delta_{g_1}, \ldots, x_n \delta_{g_n})$ be $n$-tuple with $x_i \in X_{g_i}$, we define
    \begin{equation*}
        \mathfrak{c}^{-1} := (\alpha_{g_n^{-1}}(x_n) \delta_{g_n^{-1}}, \alpha_{g_{n-1}^{-1}}(x_{n-1}) \delta_{g_{n-1}^{-1}}, \ldots, \alpha_{g_1^{-1}}(x_1) \delta_{g_1^{-1}}).
    \end{equation*}
    Then,
    \begin{enumerate}[(i)]
        \item $(\mathfrak{c}^{-1})^{-1} = \mathfrak{c}$
        \item $\mathfrak{c}$ is an $\mathcal{R}$-chain if, and only if, $\mathfrak{c}^{-1}$ is an $\mathcal{R}$-chain.
        \item $\operatorname{s}(\mathfrak{c}) = \operatorname{t}(\mathfrak{c}^{-1})$ and $\operatorname{t}(\mathfrak{c}) = \operatorname{s}(\mathfrak{c}^{-1})$.
    \end{enumerate}
    \label{l: R chain inverse}
\end{lemma}
\begin{proof}
    All assertions are obtained through direct and simple computations.
\end{proof}

Set $\epsilon : \mathfrak{C} \to G$ such that if $\mathfrak{c}= (x_1 \delta_{g_1}, \ldots, x_n \delta_{g_n})$, then
    \begin{equation*}
        \epsilon(\mathfrak{c}) = g_1 g_2 \ldots g_n.
    \end{equation*}
It is clear that for all $\mathfrak{c}, \mathfrak{o} \in \mathfrak{C}$ we have that
\begin{equation*}
    \epsilon(\mathfrak{c} \cdot \mathfrak{o}) = \epsilon(\mathfrak{c}) \epsilon(\mathfrak{o}),
\end{equation*}
and
\begin{equation}
    \epsilon(\mathfrak{c}^{-1}) = \epsilon(\mathfrak{c})^{-1}.   
    \label{eq: epsilon chain inverse}
\end{equation}

\textbf{Notation:} If $\mathcal{R}$ is an equivalence relation on a set $X$, we will use the notation $\overline{x}$ to represent the equivalence class of an element $x \in X$ (when there is no ambiguity). Similarly, if $K$ is a normal subgroup of a group $G$, we denote the equivalence class of an element $g \in G$ by $\overline{g}$.

Now we can define what we understand as a congruence on a partial action $\alpha$.

\begin{definition} \label{d: congruence on a partial action}
    Let $\alpha = (G, X, \{ X_g \}, \{ \alpha_g \})$ be a partial action of $G$ on $X$. A \textbf{partial action congruence} on $\alpha$ is a pair $(\mathcal{R}, K)$ where $\mathcal{R}$ is an equivalence relation on $X$ and $K$ is a normal subgroup of $G$, such that it satisfies the following axiom:
    
    \textbf{(PC)} If $\mathfrak{c}, \mathfrak{o} \in \mathfrak{C}_{\alpha, \mathcal{R}}$ are such that $\overline{\epsilon(\mathfrak{c})} = \overline{\epsilon(\mathfrak{o})}$ then:  $\operatorname{s}(\mathfrak{c}) \sim \operatorname{s}(\mathfrak{o}) \Leftrightarrow \operatorname{t}(\mathfrak{c}) \sim \operatorname{t}(\mathfrak{o})$.
\end{definition}

\begin{lemma} \label{l: intersection of congruences is congruence}
    Let $\alpha = (G, X, \{ X_g \}, \{ \alpha_g \})$ be a partial action, and let $\{ (\mathcal{R}_{i}, K_{i}) \}_{i \in I}$ be a family of partial action congruences on $\alpha$. Set $\mathcal{R} = \cap_{i \in I} \mathcal{R}_{i}$ and $K = \cap_{i \in I} K_i$. Then, $(\mathcal{R}, K)$ is a partial action congruence on $\alpha$.
\end{lemma}

\begin{proof}
    Observe that any $\mathcal{R}$-chain is an $\mathcal{R}_{i}$-chain for all $i \in I$. Thus, if $\mathfrak{c}$ and $\mathfrak{o}$ are $\mathcal{R}$-chains, such that $\epsilon(\mathfrak{c}) \sim_{K} \epsilon(\mathfrak{o})$, then  $\epsilon(\mathfrak{c}) \sim_{K_i} \epsilon(\mathfrak{o})$ for all $i \in I$ and
    \[
       \operatorname{s}(\mathfrak{c}) \sim_{\mathcal{R}} \operatorname{s}(\mathfrak{o}) 
       \Leftrightarrow 
       \operatorname{s}(\mathfrak{c}) \sim_{\mathcal{R}_i} \operatorname{s}(\mathfrak{o}) \, \forall i 
       \Leftrightarrow 
       \operatorname{t}(\mathfrak{c}) \sim_{\mathcal{R}_i} \operatorname{t}(\mathfrak{o}) \, \forall i
       \Leftrightarrow
       \operatorname{t}(\mathfrak{c}) \sim_{\mathcal{R}} \operatorname{t}(\mathfrak{o}). 
    \]
\end{proof}

\begin{remark}
    If $(\mathcal{R}, K)$ and $(\mathcal{L}, N)$ are two partial action congruences on $\alpha$, we write $(\mathcal{R}, K) \cap (\mathcal{L}, N)$ for $(\mathcal{R} \cap \mathcal{L}, K \cap N)$, and $(\mathcal{R}, K) \subseteq (\mathcal{L}, N)$ if $\mathcal{R} \subseteq \mathcal{L}$ and $K \subseteq N$. Thus, the set of all partial action congruences on $\alpha$ is a partially ordered set, and Lemma~\ref{l: intersection of congruences is congruence} implies that it is a meet-semilattice.
\end{remark}

A desirable property of congruences is that a morphism of partial actions determines a unique congruence that carries information about the morphism. In fact, any morphism of partial actions determines a congruence in the sense of Definition~\ref{d: congruence on a partial action}:
 
\begin{example} \label{e: morphism of partial actions determines congruences}
    Let $\alpha = (G, X, \{ X_g \}, \{ \alpha_g \})$ and $\beta = (S, Y, \{ Y_s \}, \{ \beta_s \})$ be partial actions, and let $\phi: \alpha \to \beta$ be a morphism of partial actions. Let $\mathcal{R}$ be the equivalence relation given by 
    \[
        x \sim y \Leftrightarrow \phi_0(x) = \phi_0(y),
    \]
    and $K := \ker \phi_1$, thus $(\mathcal{R}, K)$ is a partial action congruence on $\alpha$. Indeed, let $\mathfrak{c}=(x_1 \delta_{g_1}, \ldots, x_n \delta_{g_n})$ and $\mathfrak{o}=(y_1 \delta_{h_1}, \ldots, y_m \delta_{h_m})$ be $\mathcal{R}$-chains, such that $\phi_1(\epsilon(\mathfrak{c})) = \phi_1(\epsilon(\mathfrak{o}))$. Since $\mathfrak{c}$ and $\mathfrak{o}$ are $\mathcal{R}$-chains and $\phi$ is a morphism of partial actions we obtain:
    \[
        \phi_{0}(x_i) 
        = \phi_{0}(\alpha_{g_{i-1}^{-1}}(x_{i-1})) 
        = \beta_{\phi_1(g_{i-1}^{-1})}(\phi_0(x_{i-1})) \in \dom \beta_{\phi_1(g_{i-1})}, 
        \text{ for all } 1 < i \leq n 
    \]
    and
    \[
        \phi_{0}(y_i) 
        = \phi_{0}(\alpha_{h_{i-1}^{-1}}(y_{i-1})) 
        = \beta_{\phi_1(h_{i-1}^{-1})}(\phi_0(y_{i-1})) \in \dom \beta_{\phi_1(h_{i-1})}, 
        \text{ for all } 1 < i \leq m. 
    \]
    Whence,
    \[
        \phi_0(\operatorname{s}(\mathfrak{c}))
        \in \dom \beta_{\phi_1(g_1)}\beta_{\phi_1(g_2)} \ldots \beta_{\phi_1(g_n)} \subseteq \dom \beta_{\phi_1(\epsilon(\mathfrak{c}))}  
        \text{ and }
        \beta_{\phi_1(\epsilon(\mathfrak{c}))}(\phi_0(\operatorname{s}(\mathfrak{c}))) = \phi_0(\operatorname{t}(\mathfrak{c}))
    \]
    and
    \[
        \phi_0(\operatorname{s}(\mathfrak{o}))
        \in \dom \beta_{\phi_1(h_1)}\beta_{\phi_1(h_2)} \ldots \beta_{\phi_1(h_m)} \subseteq \dom \beta_{\phi_1(\epsilon(\mathfrak{o}))}  
        \text{ and }
        \beta_{\phi_1(\epsilon(\mathfrak{o}))}(\phi_0(\operatorname{s}(\mathfrak{o}))) = \phi_0(\operatorname{t}(\mathfrak{o})).
    \]
    Since $\phi_1(\epsilon(\mathfrak{c}))= \phi_1(\epsilon(\mathfrak{o}))$ and $\beta_{\phi_1(\epsilon(\mathfrak{c}))}$ is a bijection, we conclude that
    \[
        \phi_0(\operatorname{s}(\mathfrak{o})) = \phi_0(\operatorname{s}(\mathfrak{c}))
        \Leftrightarrow
        \phi_0(\operatorname{t}(\mathfrak{o})) = \phi_0(\operatorname{t}(\mathfrak{c})).
    \]
    Thus, $(\mathcal{R}, K)$ is a partial action congruence on $\alpha$.
\end{example}

We want to verify that any partial action congruence on $\alpha$ determines a \textit{quotient} partial action. Let $(\mathcal{R}, K)$ be a partial action congruence on $\alpha$. We set $\overline{G}:= G/K$ and $\overline{X}:= X / \sim_{\mathcal{R}}$. We will construct a partial action $\theta$ of $\overline{G}$ on $\overline{X}$ such that the canonical maps $\pi_0: X \to \overline{X}$ and $\pi_1: G \to \overline{G}$ determines a morphism of partial actions $\pi: \alpha \to \theta$. First, we define
\begin{equation} \label{eq: quotient partial action domain}
    D_{\overline{g}}:= \{ \overline{\operatorname{t}(\mathfrak{c})} \in \overline{X}: \mathfrak{c} \in \mathfrak{C} \text{ and } \overline{\epsilon(\mathfrak{c}}) = \overline{g} \}.
\end{equation}
Observe that by (iii) of Lemma~\ref{l: R chain inverse} and equation \eqref{eq: epsilon chain inverse} we have that
\begin{equation} \label{eq: quotient partial action map}
    D_{\overline{g}}:= \{ \overline{\operatorname{s}(\mathfrak{c})} \in \overline{X}: \mathfrak{c} \in \mathfrak{C} \text{ and } \overline{\epsilon(\mathfrak{c}}) = \overline{g^{-1}} \}.
\end{equation}
Moreover, it follows that
\begin{equation*}
    \overline{\operatorname{t}(\mathfrak{c})} \in D_{\overline{g}} \Leftrightarrow \overline{\operatorname{s}(\mathfrak{c})} \in D_{\overline{g}^{-1}},
\end{equation*}
where $\mathfrak{c} \in \mathfrak{C}$ and $\overline{\epsilon(\mathfrak{c})} = \overline{g}$.

Second, we define $\theta_{\overline{g}}: D_{\overline{g}^{-1}} \to D_{\overline{g}}$ by
\begin{equation}
    \theta_{\overline{g}}(\overline{\operatorname{s}(\mathfrak{c})}) = \overline{\operatorname{t}(\mathfrak{c})}.
\end{equation}
The map $\theta_{\overline{g}}$ is well-defined thanks to \textbf{(PC)}, $\overline{\operatorname{t}(\mathfrak{c})} \in D_{\overline{g}}$ since $\overline{\epsilon(\mathfrak{c})} = \overline{g}$. Furthermore, $\theta_{\overline{g}}$ is a bijection since 
\[
    \theta_{\overline{g}^{-1}}\big(\overline{\operatorname{t}(\mathfrak{c})}\big) 
    = \theta_{\overline{g}^{-1}}\big(\overline{\operatorname{s}(\mathfrak{c}^{-1})}\big) 
    = \operatorname{t}(\overline{\mathfrak{c}^{-1}}) = \overline{\operatorname{s}(\mathfrak{c})}.
\]

\begin{lemma} \label{l: partial action of chains}
    Let $\mathfrak{c}, \mathfrak{o} \in \mathfrak{C}$ such that $\overline{\operatorname{s}(\mathfrak{c})} = \overline{\operatorname{t}(\mathfrak{o})}$. Then 
    \begin{enumerate}[(i)]
        \item $\operatorname{t}(\mathfrak{c}) = \operatorname{t}(\mathfrak{c} \cdot \mathfrak{o})$ and $\operatorname{s}(\mathfrak{o}) = \operatorname{s}(\mathfrak{c} \cdot \mathfrak{o})$,
        \item $\overline{\operatorname{s}(\mathfrak{o})} \in D_{\overline{\epsilon(\mathfrak{c} \cdot \mathfrak{o})}^{-1}}$, $\overline{\operatorname{t}(\mathfrak{c})} \in D_{\overline{\epsilon(\mathfrak{c} \cdot \mathfrak{o})}}$,
        \item $\theta_{\overline{\epsilon(\mathfrak{o} \cdot \mathfrak{c})}}(\overline{\operatorname{s}(\mathfrak{o})}) = \theta_{\overline{\mathfrak{c}}}\theta_{\overline{\mathfrak{o}}}(\overline{\operatorname{s}(\mathfrak{o})}) = \overline{\operatorname{t}(\mathfrak{c})}$.
    \end{enumerate}
\end{lemma}

\begin{proof}
    Item $(i)$ is obvious, item $(ii)$ is a direct consequence of item $(i)$. Finally, by item $(i)$ and the definition of $\theta$ one concludes that $\theta_{\overline{\epsilon(\mathfrak{o} \cdot \mathfrak{c})}}(\overline{\operatorname{s}(\mathfrak{o})}) = \overline{\operatorname{t}(\mathfrak{c})}$, and by direct computation we have that
    \[
      \theta_{\overline{\mathfrak{c}}} \theta_{\overline{\mathfrak{o}}}(\overline{\operatorname{s}(\mathfrak{o})})
      =\theta_{\overline{\mathfrak{c}}}(\overline{\operatorname{t}(\mathfrak{o})})
      =\theta_{\overline{\mathfrak{c}}}(\overline{\operatorname{s}(\mathfrak{c})})
      = \overline{\operatorname{t}(\mathfrak{c})}.
    \]
\end{proof}

Now we want to verify all axioms of partial actions. Notice that $D_{\overline{1}} = \overline{X}$ since $\overline{x} = \overline{\operatorname{t}(x \delta_1)}$, for all $x \in X$. From this, it is clear that $\theta_{\overline{1}}(\overline{x}) = \overline{x}$. Let $\overline{z} \in D_{\overline{g}^{-1}} \cap D_{\overline{h}}$, then there exists $\mathfrak{c}, \mathfrak{o} \in \mathfrak{C}$ such that $\overline{\epsilon(\mathfrak{c})} = \overline{g}$,  $\overline{\epsilon(\mathfrak{o})} = \overline{h}$ and $\overline{\operatorname{s}(\mathfrak{c})} = \overline{\operatorname{t}(\mathfrak{o})}= \overline{z}$. Thus, by Lemma~\ref{l: partial action of chains} we conclude that $\theta_{\overline{g}}(\overline{z}) = \overline{\operatorname{t}(\mathfrak{c})} \in D_{\overline{g}} \cap D_{\overline{gs}}$. Hence, $\theta_{\overline{g}} \big( D_{\overline{g}^{-1}} \cap D_{\overline{s}} \big) \subseteq D_{\overline{g}} \cap D_{\overline{gs}}$, and the last axiom is a consequence of $(iii)$ of Lemma~\ref{l: partial action of chains}. Thus, we obtain the following theorem.

\begin{theorem} \label{t: partial action congruence}
    Let $\alpha = (G, X, \{ X_g \}, \{ \alpha_g \})$ be a partial action of $G$ on $X$, and let $(\mathcal{R}, K)$ be a partial action congruence on $\alpha$. Define $\{ D_{\overline{g}} \}_{\overline{g} \in \overline{G}}$ and $\{ \theta_{\overline{g}} \}_{\overline{g} \in \overline{G}}$ as \eqref{eq: quotient partial action domain} and \eqref{eq: quotient partial action map} respectively. Then, $\theta := (\overline{G}, \overline{X}, \{ D_{\overline{g}} \}, \{ \theta_{\overline{g}} \})$ is a partial action of $\overline{G}$ on $\overline{X}$ such that the canonical maps $\pi_0: X \to \overline{X}$ and $\pi_1: G \to \overline{G}$ determines a morphism of partial actions $\pi := (\pi_0, \pi_1): \alpha \to \theta$. 
\end{theorem}

\begin{proof}
    Only the last part of the theorem is new. Notice that $\overline{X_{g}} \subseteq D_{\overline{g}}$ since for all $x \in X_{g}$, we have that $\overline{x} = \overline{\operatorname{t}(x \delta_{g})} \in D_{\overline{g}}$, moreover, from this also follows that 
    \[
        \pi_0(\alpha_{g^{-1}}(x)) = \overline{\alpha_{g^{-1}}(x)} = \overline{\operatorname{s}(x \delta_{g})} = \theta_{\overline{g}^{-1}}(\overline{\operatorname{t}(x \delta_{g})}) =  \theta_{\pi_1(g^{-1})}(\pi_0(x)).
    \]
\end{proof}

\begin{definition}
    In the conditions of Theorem~\ref{t: partial action congruence} we say that $\theta$ is the \textbf{quotient partial action} of $\alpha$ by $(\mathcal{R},K)$ and $\pi: \alpha \to \theta$ the \textbf{quotient partial action morphism}.
\end{definition}

\begin{proposition}
    Let $\alpha = (G, X, \{ X_g \}, \{ \alpha_g \})$ and $\beta = (S, Y, \{ Y_s \}, \{ \beta_s \})$ be partial actions, and let $\phi: \alpha \to \beta$ be a morphism of partial actions. Let $(\mathcal{R}, K)$ be the partial action congruence on $\alpha$ as in Example~\ref{e: morphism of partial actions determines congruences}. We identify $\overline{X} = \im \phi_0 \subseteq Y$ and $\overline{G} = \im \phi_1 \subseteq S$, and set $\iota = (\iota_0, \iota_1)$ where $\iota_0 : \overline{X} \to Y$ and $\iota_1: \overline{G} \to S$ are the inclusion maps. Then $\iota$ is a morphism of partial actions and $\phi = \iota \circ \pi$,
    \[\begin{tikzcd}
        \alpha & \theta & \beta.
        \arrow["\pi"', two heads, from=1-1, to=1-2]
        \arrow["\iota"', hook, from=1-2, to=1-3]
        \arrow["\phi", curve={height=-12pt}, from=1-1, to=1-3]
    \end{tikzcd}\]
\end{proposition}

\begin{proof}
    We only have to prove that $\iota$ is a morphism of partial actions since the equality $\iota \circ \pi = \phi$ is clear. As we saw in Example~\ref{e: morphism of partial actions determines congruences} we have that $\overline{\operatorname{t}(\mathfrak{c})} \in Y_{\overline{\epsilon(\mathfrak{c})}}$, and $\beta_{\overline{\epsilon(\mathfrak{c})}}(\overline{\operatorname{s}(\mathfrak{c})}) = \overline{\operatorname{t}(\mathfrak{c})}$. Thus, $D_{\overline{g}} \subseteq Y_{\overline{g}}$ for all $\overline{g} \in \overline{G}$ and $\theta_{\overline{g}}(\overline{x}) = \beta_{\overline{g}}(\overline{x})$ for all $\overline{g} \in \overline{G} \subseteq S$ and $x \in D_{\overline{g}}$, so that $\iota$ is a morphism of partial actions.  
\end{proof}

\begin{proposition} \label{p: factorization of morphism of partial actions by congruences}
    Let $\phi: \alpha \to \beta$ be a morphism of partial actions, and let $(\mathcal{R}_{\phi}, K_{\phi})$ be the partial action congruence on $\alpha$ determined by $\phi$ as in Example~\ref{e: morphism of partial actions determines congruences}. If $(\mathcal{R}, K)$ is a partial action congruence on $\alpha$, such that $\mathcal{R} \subseteq \mathcal{R}_{\phi}$ and $K \subseteq K_{\phi}$. Then, there exists a unique morphism of partial actions $\varphi: \theta \to \beta$ such that $\phi = \varphi \circ \pi$, where $\theta$ is the quotient partial action of $\alpha$ by $(\mathcal{R}, K)$ and $\pi: \alpha \to \theta$ is the quotient morphism. Hence, the following is a commutative diagram in \textbf{PA}.
    \[\begin{tikzcd}
        \alpha & \beta \\
        \theta
        \arrow["\phi", from=1-1, to=1-2]
        \arrow["\pi"', from=1-1, to=2-1]
        \arrow["\varphi"', dashed, from=2-1, to=1-2]
    \end{tikzcd}\]
\end{proposition}

\begin{proof}
    Let $\alpha = (G, X, \{ X_g \}, \{ \alpha_g \})$, $\beta = (S, Y, \{ Y_s \}, \{ \beta_s \})$ and $\theta = (\overline{G}, \overline{X}, \{ D_{\overline{g}} \}, \{ \theta_{\overline{g}} \})$. If such $\phi$ there exists, then $\phi_0 = \varphi_0 \circ \pi_0$ and $\phi_1 = \varphi_1 \circ \pi_1$, therefore $\varphi_0: \overline{X} \to Y$ and $\phi_1: \overline{G} \to S$ are the unique maps that make the following diagrams commute
    \[\begin{tikzcd}
        X & Y & G & S \\
        {\overline{X}} && {\overline{G}}
        \arrow["{\phi_0}", from=1-1, to=1-2]
        \arrow["{\pi_0}"', from=1-1, to=2-1]
        \arrow["{\varphi_0}"', dashed, from=2-1, to=1-2]
        \arrow["{\pi_1}"', from=1-3, to=2-3]
        \arrow["{\phi_1}", from=1-3, to=1-4]
        \arrow["{\varphi_1}"', dotted, from=2-3, to=1-4]
    \end{tikzcd}\]
    in the categories of \textbf{Set} and \textbf{Grp} respectively. Therefore, we only have to verify that $\varphi = (\varphi_0, \varphi_1)$ is a morphism of partial actions, but this is a direct consequence of $(\mathcal{R}, K) \subseteq (\mathcal{R}_{\phi}, K_{\phi})$. 
\end{proof}

\begin{example}
    One may ask if $\phi: \alpha \to \beta$ is a morphism of partial actions such that $\phi_0$ and $\phi_1$ are surjective maps, then the quotient partial action is isomorphic to $\beta$? Contrary, to the global case, this is not always true for partial actions. For example, let $X$ be the open interval $(0,1)$,  $\alpha := \{ X, \mathbb{R}, \{ A_r \}, \{ \alpha_r \} \}$, where $A_{r}= (0, 2^{-r} \cdot \frac{1}{2})$ and $A_{-r}:=(0, \frac{1}{2})$ for $r \geq 0$, and $\alpha_r(x) = 2^{r} \cdot x$ for all $r \in \mathbb{R}$. Analogously we define $\beta := \{ X, \mathbb{R}, \{ B_r \}, \{ \beta_r \} \}$, where $B_{r}= (0, 2^{-r})$ and $A_{-r}:=(0, 1)$ for $r \geq 0$, and $\beta_r(x) = 2^{r} \cdot x$ for all $r \in \mathbb{R}$. Thus, the identity maps on $X$ and $\mathbb{R}$ determine a morphism of partial actions $\alpha \to \beta$, the partial action congruence determined by the inclusion is $(\varnothing, \{ 1 \})$, thus the quotient partial action is just $\alpha$, but it is clear that $\alpha$ and $\beta$ are not isomorphic.
\end{example}

 \textbf{Notation:} We denote the empty set and the empty map by $\varnothing$ and $\emptyset$ respectively.

 \begin{remark} \label{r: PA has initial and terminal object}
     Notice that the category \textbf{PA} has initial and terminal objects. Indeed, it is easy to verify that the action of the group with one element $\{ 1 \}$ acting on a one-point set a terminal object of \textbf{PA}. Since, partial actions admit empty sets as domains we conclude that the group $\{ 1 \}$ acts on the empty set $\varnothing$ via the $\varnothing$ map, and this partial action is the initial object of \textbf{PA}.
 \end{remark}

\subsection{The category \textbf{PA} is complete and cocomplete}

The main objective of this section is to prove the following:
\begin{theorem}
    The category \textbf{PA} is complete and cocomplete.
\end{theorem}

We will show that the category \textbf{PA} has products, coproducts, equalizers and coequalizers. 

\subsubsection{Completeness}

Let $\{ \alpha^i \}_{i \in I}$ be a family of partial actions, where 
\[
   \alpha^{i}:=(X^i, G^i, \{ X^i_{g^i} \}_{g^i \in G^i}, \{ \alpha^i_{g^i} \}_{g^i \in G^i} ).
\]
We define:   
\begin{enumerate}[(i)]
    \item $G = \prod_{i \in I} G^i$, 
    \item $X = \prod_{i \in I} X^i$,
    \item for $(g^i)_{i \in I} \in G$ we set $X_{(g^i)}:= \prod_{i \in I} X^i_{g^i}$,
    \item for $(x^i)_{i \in I} \in X_{(g^i)_{i \in I}^{-1}}$, define $\alpha_{(g^i)_{i \in I}}((x^i)_{i \in I}) := (\alpha_{g^i}(x^i))_{i \in I}$.
\end{enumerate}
It is clear that $\alpha$ is a partial action of $G$ on $X$, for all $i \in I$ we define $\pi^i: \alpha \to \alpha^i$, determined by the natural maps $G \to G^i$ and $X \to X^i$, then $\pi^i$ is a morphism of partial actions. Let $\theta := (Y, S, \{ Y_s \}, \{ \theta_s \})$ be a partial action and $\{ \phi^i: \theta \to \alpha^i \}_{i \in I}$ a family of morphisms of partial actions. Set
\[
    \psi_0: Y \to X, \text{ such that } \psi_0(y) = (\phi_0^i(y))_{i \in I}
\]
and
\[
    \psi_1: S \to G, \text{ such that } \psi_1(s) = (\phi_1^i(s))_{i \in I}.
\]
Then $\psi:= (\psi_0, \psi_1): \theta \to \alpha$ is a morphism of partial actions such that $\phi^i = \phi \circ \pi^i$. Henceforth, $\alpha$ is the product of the family $\{ \alpha^i \}$. Thus, we have the following proposition.

\begin{proposition} \label{p: PA product}
    The category \textbf{PA} has products.
\end{proposition}
Let $\phi: \alpha \to \beta$ and $\varphi: \alpha \to \beta$ morphism of partial group actions, where $\alpha = (G, X, \{ X_g \}, \{ \alpha_g \})$, $\beta = (S, Y, \{ Y_s \}, \{ \beta_s \})$. Define:
\begin{enumerate}[(i)]
    \item $E := \{ g \in G : \phi_1(g) = \varphi(g) \}$, 
    \item $Z := \{ x \in X : \phi_0(x) = \varphi_0(x) \}$,
    \item $Z_g := X_g \cap Z$ for $g \in E$,    
    \item $\theta_g := \alpha_g|_{Z_{g^{-1}}}$ for $g \in E$.
\end{enumerate}
Notice that if $x \in Z_{g^{-1}}$, $g \in E$, then
\[
    \phi_0(\alpha_g(x)) = \beta_{\phi_1(g)}(\phi_0(x)) = \beta_{\varphi_1(g)}(\varphi_0(x)) = \varphi_0 (\alpha_g(x)).
\]
Hence, $\alpha_{g}(x) \in Z_g$ for all $g \in E$ and $x \in Z_{g^{-1}}$. Thus, $\theta := (Z, E, \{ Z_g \}_{g \in E}, \{ \theta_g \}_{g \in E})$ is a well-defined partial action of $E$ on $Z$. Furthermore, since $Z$ is the equalizer of $\phi_0, \varphi_0 : X \to Y$ and $E$ is the equalizer of $\phi_1, \varphi_1 : G \to S$, it is clear that $\theta$ is the equalizer of $\phi$ and $\varphi$, where the associated map is the inclusion $\iota: \theta \to \alpha$.

\begin{proposition} \label{p: PA equalizer}
    The category \textbf{PA} has equalizers.
\end{proposition}

\begin{corollary}
    The category \textbf{PA} is complete.
\end{corollary}

\begin{proof}
    By Proposition~\ref{p: PA product} and Proposition~\ref{p: PA equalizer} the category \textbf{PA} has products and equalizers, then by \cite[Proposition 5.1.26]{LBCT} we conclude that \textbf{PA} is complete.
\end{proof}

\subsubsection{Cocompleteness}

Let $\{ \alpha^i \}_{i \in I}$ be a family of partial actions, where 
\[
   \alpha^{i}:=(X^i, G^i, \{ X^i_{g^i} \}_{g^i \in G^i}, \{ \alpha^i_{g^i} \}_{g^i \in G^i} ).
\]
We set $G = \coprod_{i \in I} G^i$ and $X = \coprod_{i \in I} X^i$. Recall that the coproduct in the category of groups is the free product. We identify $G^i$ with its natural inclusion into $G$ and $X^i$ with its inclusion into $X$.
\begin{enumerate}[(i)]
    \item We set $X_1 = X$ and $\theta_1 = 1_X$,
    \item For all $g \in G$ we define
        \[
           X_{g} := \left\{\begin{matrix}
              X_{g}^{i}  & \text{ if } g \in G^{i} \text{ for some } i \in I\\ 
               \varnothing & \text{ otherwise}
            \end{matrix}\right.
        \]

    \item For all $g \in G$ we define
        \[
           \alpha_{g} := \left\{\begin{matrix}
              \alpha_{g}^{i}  & \text{ if } g \neq 1, \text{ and } g \in G^{i} \text{ for some } i \in I\\ 
               \emptyset      & \text{ otherwise}.
            \end{matrix}\right.
        \]
\end{enumerate}
Then $\theta$ is a partial action of $G$ on $X$. Furthermore, since $G$ and $X$ are the coproducts of $\{ G^i \}$ and $\{ X^i \}$ in their respective categories then $\theta$ is the coproduct of $\{ \alpha^i \}$ where the maps $\iota_i :\alpha_i \to \theta$ are determined by the natural inclusion.

\begin{proposition} \label{p: PA has coproducts}
    The category \textbf{PA} has coproducts.
\end{proposition}

Now we want to verify the existence of coequalizers in \textbf{PA}. Let $\phi: \alpha \to \beta$ and $\varphi: \alpha \to \beta$ be morphisms of partial actions, where $\alpha = (G, X, \{ X_g \}, \{ \alpha_g \})$, $\beta = (S, Y, \{ Y_s \}, \{ \beta_s \})$. Set 
\[
    W:= \{ \psi: \beta \to \gamma : \psi \text{ is a morphism of partial actions such that } \psi \circ \phi = \psi \circ \varphi \}.
\]
We define $\mathfrak{W}$ as the set of partial action congruences induced by the elements of $W$, i.e., $(\mathcal{R}, K) \in \mathfrak{W}$ if, and only if, $(\mathcal{R}, K)$ is a partial action congruence and there exists $\psi \in W$ such that $\psi$ determines $(\mathcal{R}, K)$ in the sense of Example~\ref{e: morphism of partial actions determines congruences}. Notice that by Remark~\ref{r: PA has initial and terminal object} we have that $W$ is not empty, thus $\mathfrak{W} \neq \varnothing$. 

Let $\mathcal{C}$ be the partial action congruence on $\beta$ such that $\mathcal{C} = \cap_{C \in \mathfrak{W}} C$, i.e., $\mathcal{C}$ is the intersection of all the congruences in $\mathfrak{W}$. Let $\pi^\mathcal{C}: \beta \to \theta^\mathcal{C}$ be the quotient morphism of partial actions, where $\theta^\mathcal{C}$ is the quotient partial action of $\beta$ by $\mathcal{C}$. Note that since $\mathcal{C}$ is the intersection of all the congruences in $\mathfrak{W}$, then $\pi_j^\mathcal{C} \circ \phi_j = \pi_j^\mathcal{C} \circ \varphi_j$, for $j \in \{ 0, 1 \}$. Thus, by Remark~\ref{r: equality of partial actions}, $\pi^{\mathcal{C}} \circ \phi = \pi^{\mathcal{C}} \circ \varphi$. Finally, by Proposition~\ref{p: factorization of morphism of partial actions by congruences} we have that $\theta^\mathcal{C}$ is the coequalizer of $\phi$ and $\varphi$.

\begin{proposition} \label{p: PA has coequalizer}
    The category \textbf{PA} has coequalizers.
\end{proposition}

\begin{corollary}
    The category \textbf{PA} is cocomplete.
\end{corollary}

\begin{proof}
    By Proposition~\ref{p: PA has coproducts}, Proposition~\ref{p: PA has coequalizer} and the dual statement of \cite[Proposition 5.1.26]{LBCT} we conclude that \textbf{PA} is cocomplete.
\end{proof}


\section{Partial actions and groupoids}

We will use the usual notation for groupoids, i.e., $\Gamma {\rightrightarrows} X$ denotes a groupoid with morphism $\Gamma$ and set of objects $X$. Usually by abuse of notation if there is no ambiguity we will denote the groupoid $\Gamma {\rightrightarrows} X$ just by $\Gamma$. Moreover, $\operatorname{t}:\Gamma \to X$ and $\operatorname{s}: \Gamma \to X$ denotes the usual target and source maps. \underline{In this section}, the letters $\alpha$ and $\beta$ will be used to denote the morphisms of a groupoid $\Gamma$ and do not represent partial actions anymore.

\subsection{From partial actions to groupoids}

Groupoids are small categories whose every morphism is an isomorphism. This short definition is equivalent to the following one (Definition 3.1 \cite{IRGGd}): A groupoid $\Gamma \rightarrow X$ consists of a set $\Gamma$, a set $X$, and two maps $s,t : \Gamma \rightarrow X$, called the source and target map, respectively. The set $X$ is called the space of objects of the groupoid. The elements in $\Gamma$ will be called the morphisms or arrows of the groupoid. Two morphisms $\beta, \alpha \in \Gamma$ will be said to be composable if $t(\alpha) = s(\beta)$. The set of composable pairs is denoted by $\Gamma^{(2)}$. The sets $\Gamma$ and $X$ are equipped with a map $\cdot : \Gamma^{(2)} \rightarrow \Gamma$, such that:
\begin{enumerate}
    \item (Associativity) $(\gamma \cdot \beta) \cdot \alpha = \gamma \cdot (\beta \cdot \alpha)$, whenever $(\beta, \alpha)$ and $(\gamma, \beta)$ are in $\Gamma^{(2)}$.
    \item (Units) For any object $x \in \Omega$, there exists a morphism $1_x : x \rightarrow x$, called the unit or identity at $x$, such that if $\alpha : x \rightarrow y$, then $\alpha \cdot 1_x = \alpha$ and $1_y \cdot \alpha = \alpha$.
    \item (Inverse) For any morphism $\alpha : x \rightarrow y$, there exists another morphism $\beta : y \rightarrow x$, such that $\beta \cdot \alpha = 1_x$ and $\alpha \cdot \beta = 1_y$. Such morphism will be denoted as $\alpha^{-1}$.
\end{enumerate}
We say that a groupoid is \textbf{connected} (or transitive) if for any pair of elements $x, y \in X$, there exists a morphism $\gamma \in \Gamma$ such that $\operatorname{s}(\gamma) = x$ and $\operatorname{t}(\gamma) = y$, equivalently $\hom_{\Gamma}(x, y) \neq \varnothing$.

Let $\theta:= (G, X, \{ X_{g} \}_{g \in G}, \{ \theta_{g} \}_{g \in G})$ be a partial action of a group $G$ on a set $X$. Then, we can obtain a groupoid $\Gamma_{\theta} {\rightrightarrows} X$ as follows: the set of objects is $X$, and we define the morphisms 
\[
    \Gamma_{\theta}:=\left\{(\theta_g(x), g, x) : g \in G \text{ and } x \in X_{g^{-1}} \right\},
\]
with the composition given by 
 \begin{equation} \label{eq: partial action groupoid composition}
     (\theta_g(y), g , y)(\theta_{h}(x),h,x) = (\theta_{gh}(x), gh, x) \Longleftrightarrow x \in \dom \theta_g \theta_h \text{ and } y = \theta_{h}(x).
 \end{equation}
 It is easy to verify that $\Gamma_{\theta}$ is in fact a groupoid, such that the source and target maps are given by
 \[
     \operatorname{s}(\theta_{g}(x), g, x) = \theta_{g}(x) \text{ and } \operatorname{t}((\theta_{g}(x), g, x) = x,
 \]
 and with units $\{ (x, 1, x) \}_{x \in X}$. This groupoid is known as the \textbf{partial action groupoid} of $\theta$. Furthermore, if $\theta$ and $\rho$ are partial actions, and $\phi=(\phi_0, \phi_1) : \theta \to \rho$ is a morphism of partial actions, then the map 
 \begin{equation} \label{eq: induced partial action groupoid morphism}
 	\Psi(\phi): \Gamma_{\theta} \to \Gamma_{\rho} \text{ such that } \Psi(\phi)(y,g,x):= (\phi_0(y), \phi_1(g), \phi_0(x))
 \end{equation}
 is a morphism of groupoids. Therefore, if we set $\Psi(\theta) := \Gamma_{\theta}$ we obtain a functor $\Psi: \textbf{PA} \to \textbf{Grpd}$. It is clear that partial action groupoids are a direct generalization of action groupoids (see, for example, \cite[Definition 4.7]{IRGGd}). Partial action groupoids were previously studied in a continuous context in \cite{FAPAAG}.

 We can characterize isomorphism of partial actions using their respective associated groupoids:

 \begin{lemma} \label{l: isomorphism of partial actions via groupoids}
     Let $\theta$ and $\theta'$ be partial actions, and $f: \theta \to \theta'$ a morphism such that $f_{0}$, $f_{1}$ are isomorphism. Then, $\Psi(f): \Psi(\theta) \to \Psi(\theta')$ is an isomorphism of groupoids if, and only if, $f$ is an isomorphism of partial actions.
 \end{lemma}
 \begin{proof}
     Observe that if
    \[
        (\theta_{g}(x), g, x) \mapsto (f_{0}(\theta_{g}(x)), f_1(g), f_0(x))=(\theta'_{f_1(g)}(f_0(x)), f_1(g), f_0(x))
    \]
     is an isomorphism of groupoids, then $x \in \dom \theta_{g}$ if, and only if, $f_{0}(x) \in \dom \theta'_{f_{1}(g)}$. This implies that $(f_{0}^{-1}, f_{1}^{-1})$ is a morphism of partial group actions. Consequently, $f$ is an isomorphism of partial actions.
 \end{proof}

 \textbf{Notation:} Sometimes for the sake of simplicity we will denote the elements of the partial action groupoid $\Gamma_{\theta}$ by
 \begin{equation}
    g_{x} := (\theta_{g}(x), g , x) \in \Gamma_{\theta}.
 \end{equation}
 In this case \eqref{eq: partial action groupoid composition} takes the form:  
 \begin{equation} \label{eq: composition rule simplified partial action groupoid}
     g_{y} h_{x} = (gh)_{x} \Longleftrightarrow \theta_{h}(x) = y,
 \end{equation}
 and \eqref{eq: induced partial action groupoid morphism} takes the form:
 \begin{equation}
     \Psi(\phi)(g_{x}) := \phi_1(g)_{\phi_0(x)}.
 \end{equation}
We will employ one notation or the other as deemed appropriate.    

\begin{lemma} \label{l: coproduct commutes with Psi}
    Let $\{ \theta^{i} \}$ be a family of partial actions, then $\Psi\big(\coprod_{i} \theta^{i}\big) = \bigsqcup_{i} \Psi(\theta^{i})$.
\end{lemma}
\begin{proof}
    Let $\theta = \coprod_{i} \theta^{i}$, then by definition of $\Gamma_{\theta}$ we get
    \begin{align*}
        \Gamma_{\theta}
        &= \big\{ (\theta_g(x), g, x ) : g \in G \text{ and } x \in X_{g^{-1}} \big\} \\
        &= \big\{ (\theta^{i}_g(x), g, x ) : g \in G^{i} \text{ for some } i \text{ and } x \in X_{g^{-1}}^{i} \big\} \\
        &= \bigsqcup_{i} \Gamma_{\theta^{i}}.
    \end{align*}
\end{proof}

 \subsection{From groupoids to partial actions}

It is natural to ask whether a groupoid arises from a partial action. It is well-known that if the groupoid $\Gamma {\rightrightarrows} X$ is connected, then there exists a group $G$ with a global action $\theta$ on $X$ such that $\Gamma_{\theta} {\rightrightarrows} X \cong \Gamma {\rightrightarrows} X$ (such a global action is not necessarily unique). For the sake of completeness, we provide a brief proof here:

\begin{proposition} \label{p: connected groupoids are global actions}
    Let $\Gamma \rightrightarrows X$ be a connected groupoid. Then, there exists a group $G$ and a global action $\theta$ of $G$ on $X$ such that $\Gamma_{\theta} \cong \Gamma$.
\end{proposition}
\begin{proof}
    Let $T \rightrightarrows X$ be the tree groupoid on the set $X$, i.e. the groupoid such that for all $x, y \in X$ there exists a unique morphism $x \to y$ in $T$. Let $G$ be the isotropy group of an arbitrary element of $\Gamma$. Since, $\Gamma$ is connected, then by \cite[Proposition 5.7]{IRGGd} we have that $\Gamma \rightrightarrows X$ is isomorphic to the product groupoid $G \times T \rightrightarrows X$. Let $\star : X \times X \to X$ a group structure on $X$, define $\theta: G \times X \to X$ such that $\theta_{(g, x)}(y) := x \star y$. Then, $\theta$ is a well-defined global action of $G \times X$ on $X$.
    Furthermore, $\Gamma_{\theta} \rightrightarrows X$ is a connected groupoid with isotropy group $G$. Thus, by \cite[Proposition 5.7]{IRGGd}, it is also isomorphic to $G \times T \rightrightarrows X$, what concludes our proof.
\end{proof}

In general, this is not true for arbitrary groupoids. For example, a groupoid consisting of two disconnected objects such that the isotropy group of those objects are not isomorphic cannot be obtained from a global action. But we can obtain any groupoid $\Gamma$ from a partial action as we will see in the next proposition: 

\begin{proposition} \label{p: groupoids are partial actions}
    Let $\Gamma {\rightrightarrows} X$ be a groupoid, then there exists a group $G$ and partial action $\theta$ of $G$ on $X$ such that $\Psi(\theta) \cong \Gamma$.
\end{proposition}

\begin{proof}
    Let $\Gamma$ be a groupoid and $\{ \Gamma_{i} \}_{i \in I}$ be the set of its connected components. By Proposition~\ref{p: connected groupoids are global actions}, for each $i \in I$, we can choose a global action $\theta^i$ such that $\Gamma_{\theta^i} \cong \Gamma_{i}$. Let $\theta := \coprod_{i \in I} \theta^i$. Then, by Lemma~\ref{l: coproduct commutes with Psi}, we have:
    \[
       \Gamma_{\theta} = \bigsqcup_{i \in I} \Gamma_{\theta_{i}} \cong \bigsqcup_{i \in I} \Gamma_i = \Gamma,
    \]
    since $\Gamma_{\theta^{i}} \cong \Gamma_{i}$. Hence, $\Gamma \cong \Gamma_{\theta}$.
\end{proof}

An interesting motivation for studying the relationship between groupoids and partial group actions is the observation that any global action of a groupoid on a set corresponds to a partial action equipped with a morphism of partial actions. We recall the definition of a groupoid action on a set from \cite[Definition 4.8]{IRGGd}:

\begin{definition} \label{d: groupoid action}
    Let $\Gamma {\rightrightarrows} X$ be a groupoid and $\mu: Y \to X$ a map of sets. Define the composable set $\Gamma \circ Y = \{ (\gamma, y) \in \Gamma \times Y : \operatorname{s}(\gamma) = \mu(y) \}$. An action of $\Gamma$ on $Y$ is a map $\Lambda : \Gamma \circ Y \to Y$, such that the family of maps $\Lambda_{\gamma} := \Lambda(\gamma, -): \mu^{-1}(\operatorname{s}(\gamma)) \to Y$, satisfy:
    \begin{enumerate}[(i)]
        \item $\mu(\Lambda_{\gamma}(y)) = \operatorname{t}(\gamma)$,
        \item $\Lambda_{1_{x}}(y) = y$ for all any $y \in \mu^{-1}(x)$, and
        \item $\Lambda_{\alpha} \Lambda_{\beta} = \Lambda_{\alpha \beta} (y)$ for $\alpha, \, \beta \in \Gamma$ such that $\operatorname{s}(\alpha) = \operatorname{t}(\beta)$.
    \end{enumerate}
\end{definition}

\begin{proposition} \label{p: groupoids actions are partial actions}
    Let $\Gamma {\rightrightarrows} X$ be a groupoid, and $\Lambda: \Gamma \circ Y \to Y$ an action of $\Gamma$ on $Y$ with momentum map $\mu: Y \to X$. Let $\theta:= \big( G, X, \{ X_{g} \}, \{ \theta_{g} \} \big)$ be a partial action such that $\Gamma_{\theta} \cong \Gamma$, identify $\Gamma$ with $\Gamma_{\theta}$. Define $\hat{\theta} := \big( G, Y, \{ Y_{g} \}, \{ \hat{\theta}_{g} \} \big)$ of $G$ on $Y$ such that:
\begin{enumerate}[(i)]
    \item $Y_{g} := \mu^{-1} \big( X_{g} \big)$,
    \item $\hat{\theta}_{g}(y) := \Lambda \big( g_{\mu(y)},  y \big)$.
\end{enumerate}
    Then $\hat{\theta}$ is a partial group action of $G$ on $Y$. Furthermore, $\mu$ determines a morphism of partial actions $(\mu, id_{G}): \hat{\theta} \to \theta$.
\end{proposition}

\begin{proof}
    First notice that if $y \in Y_{g^{-1}}$, then $\Lambda(g_{\mu(y)}, y)$ is well-defined since $\operatorname{s}(g_{\mu(y)}) = \mu(y)$. Furthermore, by $(i)$ of Definition~\ref{d: groupoid action} we have that
    \[
        \mu(\hat{\theta}_{g}(y)) = \mu(\Lambda(g_{\mu}(y), y)) = \theta_{g}(\mu(y)) \in X_{g}. 
    \]
    Thus, $\hat{\theta}_{g}(Y_{g^{-1}}) \subseteq Y_{g}$ and the maps $\hat{\theta}_{g}: Y_{g^{-1}} \to Y_{g}$ are well-defined.
    
    Observe that $Y_{1} = \mu^{-1}(X) = Y$, and by $(i)$ of Definition~\ref{d: groupoid action} we have that
    \[
        \hat{\theta}(y) = \Lambda(1_{\mu(y)}, y) = y, \text{ for all } y \in Y.
    \]
    Let $y \in Y_{g^{-1}} \cap Y_{h}$, therefore $\mu(y) \in X_{g^{-1}} \cap X_{h}$, thus 
    \[
        \operatorname{t}(g_{\mu(y)})= \operatorname{t}(\theta_{g}(\mu(y)), g, \mu(y)) = \theta_{g}(\mu(y)) \in X_{gh}.
    \]
    Then, by $(i)$ of Definition~\ref{d: groupoid action} we have that
    \[
        \hat{\theta}_{g}(y) = \Lambda(g_{\mu(y)},y) \in \mu^{-1}(X_{gh}) = Y_{gh}.
    \]
    Therefore, $\hat{\theta}_{g}(Y_{g^{-1}} \cap Y_{h}) \subseteq Y_{h}$. For the last a condition of partial action take any $y \in \hat{\theta}_{h}^{-1}(Y_{g^{-1}} \cap Y_{h})$ we have that
    \begin{align*}
        \hat{\theta}_{g} \hat{\theta}_{h}(y) 
        &=  \hat{\theta}_{g} \big( \Lambda(h_{\mu(y)}, y) \big) \\
        &=  \Lambda\big( g_{\mu(\Lambda(h_{\mu(y)}, y))}, \Lambda(h_{\mu(y)}, y) \big) \\
        &=  \Lambda\big( g_{\operatorname{t}(h_{\mu(y)})}, \Lambda(h_{\mu(y)}, y) \big) \\
        (\flat)  &=  \Lambda\big( g_{\theta_{h}(\mu(y))}, \Lambda(h_{\mu(y)}, y) \big) \\
        &=  \Lambda\big( g_{\theta_{h}(\mu(y))}h_{\mu(y)}, y) \big) \\
        (\flat \flat) &=  \Lambda\big( (gh)_{\mu(y)}, y) \big) \\
        &=  \hat{\theta}_{gh}(y),
    \end{align*}
    where $(\flat)$ and $(\flat \flat)$ are due $(iii)$ of Definition~\ref{d: groupoid action} and \eqref{eq: composition rule simplified partial action groupoid} respectively. Finally, notice that $\mu$ determines a morphism of partial actions since  by definition $\mu(Y_{g}) \subseteq X_{g}$, and by $(i)$ of Definition~\ref{d: groupoid action} we have that
    \[
        \mu \big(\hat{\theta}_{g}(y) \big) = \mu\big( \Lambda(g_{\mu(y)}, y) \big) = \operatorname{t}(g_{\mu(y)}) = \theta_{g}(\mu(y)). 
    \]
\end{proof}

\begin{remark}
    Observe that in the conditions of Proposition~\ref{p: groupoids are partial actions} we have that $\mu: Y \to X$ is a \textit{strong} $G$-equivariant map, i.e., $\mu^{-1}(X_g) \subseteq Y_g$. Furthermore, the partial group actions $\theta$, $\hat{\theta}$ and the $G$-equivariant map $\mu$ contains all the information about the groupoid action $\Lambda$.
\end{remark}


Proposition~\ref{p: groupoids are partial actions} and Proposition~\ref{p: groupoids actions are partial actions} serve as a clear motivation to understand the relationship between the categories \textbf{PA} and \textbf{Grpd}. One may begin by asking if we can construct a functor from \textbf{Grpd} to \textbf{PA}. To accomplish this we need a canonical way to obtain a partial action $\theta^{\Gamma}$ form a groupoid $\Gamma$. We will show that for every groupoid $\Gamma {\rightrightarrows} X$ there exists a unique partial action $\theta^{\Gamma}$ (up to isomorphism) of a group $G_{\Gamma}$ on $X$ such that $\Psi(\theta^{\Gamma}) \cong \Gamma$ and that satisfies certain universal property. The main objective of this section is to prove the following theorem:
\begin{theorem} \label{t: groupoids are partial actions}
    There exists a functor $\Phi: \textbf{Grpd} \to \textbf{PA}$ such that $\Psi \Phi \cong 1_{\textbf{Grpd}}$. Furthermore, if $\Gamma$ is a groupoid, then $\Phi(\Gamma)$ is the unique (up isomorphism) partial action such that:
    \begin{enumerate}[(i)]
        \item there exists a natural isomorphism of groupoids $\eta_{\Gamma} : \Psi(\Phi(\Gamma)) \to \Gamma$, and 
        \item if $\theta$ is a partial action and $f: \Psi(\theta) \to \Gamma$ an isomorphism of groupoids, then there exists a unique morphism of partial actions $\phi: \Phi(\Gamma) \to \theta$ such that the following diagram commutes
            \[\begin{tikzcd}
                \Gamma & {\Psi(\theta)} \\
                {\Psi(\Phi(\Gamma)).}
                \arrow["f", from=1-1, to=1-2]
                \arrow["{\Psi(\phi)}"', from=2-1, to=1-2]
                \arrow["{\eta_{\Gamma}}", from=2-1, to=1-1]
            \end{tikzcd}\]
        In particular $\Psi(\theta): \Psi(\theta^{\Gamma}) \to \Psi(\theta)$ is an isomorphism of groupoids.
    \end{enumerate}
    Moreover, the functors $\Phi$ and $\Psi$ determine an adjunction between the categories \textbf{PA} and \textbf{Grpd}, i.e., for any partial action $\theta$ and any groupoid $\Gamma$ there exists a natural isomorphism
    \[
        \hom(\Phi(\Gamma), \theta) \cong \hom(\Gamma, \Psi(\theta)).
    \]
\end{theorem}

\subsection{The universal partial action group of a groupoid}

\underline{From now on}, we fix a groupoid $\Gamma {\rightrightarrows} X$. 

\begin{definition} \label{d: universal partial action group of a groupoid}
    We define the \textbf{universal partial action group} of $\Gamma \rightrightarrows X$ as the group $G_\Gamma$ generated by the set of symbols $\{ \lceil \gamma \rceil \}_{\gamma \in \Gamma}$ subject to the relations:
    \begin{enumerate}[(i)]
        \item $\lceil 1_x \rceil = 1$ for all $x \in X$;
        \item $\lceil \alpha \rceil \lceil \beta \rceil = \lceil \alpha \beta \rceil$ for all $(\alpha, \beta) \in \Gamma^{(2)}$.
    \end{enumerate}
\end{definition}
 
The objective of this section is to characterize some important properties of $G_{\Gamma}$ that will allow us to construct the universal partial group action of a groupoid.

Set $F$ as the free group whose basis is the set of symbols $\{ \lfloor \alpha \rfloor: \alpha \in \Gamma \}$. 
Consider the following set $\Gamma \cup \{ 1_{X}, \emptyset \}$, where $\emptyset$ denotes the empty function $\emptyset: \varnothing \to \varnothing$. Then, $\Gamma \cup \{ 1_X, \emptyset \}$ forms an inverse monoid (for more details on inverse semigroups, see \cite{lawson1998inverse}) with the product defined as follows:
\begin{equation*}
   1_X w = w 1_X = w \text{ for all } w \in \Gamma \cup \{ 1_X, \emptyset \}
\end{equation*}
\begin{equation*}
   \emptyset w = w \emptyset = \emptyset \text{ for all } w \in \Gamma \cup \{ 1_X, \emptyset \}
\end{equation*}
\begin{equation*}
    \alpha \beta = \left\{\begin{matrix}
       \alpha \beta \in \Gamma & \text{ if } \operatorname{s}(\alpha) = \operatorname{t}(\beta), \\ 
      \emptyset & \text{otherwise,} 
      \end{matrix}\right.
\end{equation*}

For any set of elements $\{ \alpha_{i} \}_{i = 1}^{n}$ of $\Gamma$, if the product $\alpha_{1} \alpha_{2} \ldots \alpha_{n}$ does not exist in $\Gamma$, we write $\alpha_{1} \alpha_{2} \ldots \alpha_{n} = \emptyset$, i.e., we say that the product of no-composable elements of $\Gamma$ is $\emptyset$. 

Now we define the function $\pi: F \to \Gamma \cup \{ 1_{X}, \emptyset \}$ such that \begin{equation*}
    \pi(w) = \left\{\begin{matrix}
 1_X & \text{ if } w = 1_F, \\ 
  \alpha_{1}^{i_{1}} \alpha_{2}^{i_{2}} \ldots \alpha_{n}^{i_{n}} & \text{otherwise,} 
\end{matrix}\right.
\end{equation*}
where 
\[
   w = \lfloor \alpha_{1} \rfloor^{i_{1}} \ldots \lfloor \alpha_{n} \rfloor^{i_{n}}, \text{ with } i_{j} \in \{ 1, -1 \},
\]
is the reduced word form of $w \in F$. 

Note that the elements of the form
\begin{equation} \label{eq: basis elements of F}
    \lfloor \alpha \rfloor^{i} \text{ such that } \alpha \in \Gamma \text{ and } i \in \{ 1,-1\} 
 \end{equation}
generates $F$ as a group.
 \begin{remark} \label{r: easy remark}
    Notice that given any pair of elements $a,b$ of the form \eqref{eq: basis elements of F}, we have that if $ab=1_F$, then $\pi(a) \pi(b) = 1_{\operatorname{s}(\pi(b))}$.
 \end{remark}

 Let $a, b \in \Gamma \cup \{ 1_X, \emptyset \}$ we say that
 \begin{align} \label{eq: orden in groupoid semigroup}
    a \leq b 
     &\Leftrightarrow \text{ there exists } u \in \{ 1_x : x \in X \} \cup \{ 1_X, \emptyset \} \text{ such that } a = b u \\
     &\Leftrightarrow \text{ there exists } u \in \{ 1_x : x \in X \} \cup \{ 1_X, \emptyset \} \text{ such that } a = ub. \notag
 \end{align}
 The relation $\leq$ defines a partial order on the inverse monoid $\Gamma \cup \{ 1_X, \emptyset \}$, which is, in fact, the natural order of an inverse semigroup (see \cite{lawson1998inverse}).

 \begin{lemma} \label{l: pi is a partial representation}
  Let $w,v \in F$. Then, $\pi(w)\pi(v) \leq \pi(wv)$.
 \end{lemma}

 \begin{proof}
  Let $w = a_1 \ldots a_n$, and $v = b_1 \ldots b_m$, be elements of $F$ in their reduced form, where $a_i$ and $b_i$ are of the form \eqref{eq: basis elements of F}. Then, there exists $0 \leq l \leq \min(n,m)$ such that
  \[ 
    wv = a_1 \ldots a_{n-l} b_{l+1} \ldots b_m
  \]
  is the reduced form of $wv.$ Thus, by Remark \ref{r: easy remark} we obtain
  \begin{align*}
    \pi(w) \pi(v) &= \alpha_1\alpha_2 \ldots \alpha_n \beta_1 \beta_2 \dots \beta_m \\ 
                  & \leq \alpha_1\alpha_2 \ldots \alpha_{n-l} \beta_{l+1}\beta_2 \dots \beta_m \\
                  &= \pi(wv),
  \end{align*}
  where $\pi(a_i)=\alpha_i$ and $\pi(b_j)=\beta_j$.
 \end{proof}

 An element $A \in F$ such that $\pi(A) \neq \emptyset$ will be called a \textbf{path} of $F$. Given a path $A$ we define
 \[
    S(A):= \operatorname{s}(\pi(A)) \text{ and } T(A):=\operatorname{t}(\pi(A)).
 \]
 
 \begin{lemma} \label{l: product of paths target and source}
   Let $A$ and $B$ be two paths, then $AB$ is a path if, and only if, $S(A)=T(B)$. Furthermore, if $AB$ is a path and $AB \neq 1_F$, then $T(AB)=T(A)$ and $S(AB)=S(B)$. 
 \end{lemma}
 \begin{proof}
  Suppose that $A = a_1 \ldots a_n$ and $B=b_1 \ldots b_m$ are reduced representations. Then, the first part is clear, since $\pi(AB)\neq \emptyset$ if, and only if, $\pi(a_n)\pi(b_1)\neq \emptyset$. The second part is a consequence of Lemma~\ref{l: pi is a partial representation} that gives the equation:
   \[
    \emptyset \neq  \alpha_1 \ldots \alpha_n \beta_1 \dots \beta_m = \pi(A)\pi(B)=\pi(AB) \neq 1_F.
   \]
 \end{proof}

 If $A$ and $B$ are paths we say that $A$ is \textbf{unlinked} to $B$ if $\pi(AB) = \emptyset$, or equivalently $\operatorname{S}(A) \neq \operatorname{T}(B)$. 

\textbf{Notation:} If $A$ is unlinked to $B$, we write $A\dagger B$; otherwise, if $AB$ forms a path, we write $\underline{AB}$.

 The \textbf{path representation} of an element $w \in F$ is a product of consecutive unlinked paths $A_1, A_2, \ldots, A_n$ ($A_i\dagger A_{i+1}$, for $i = 1, \ldots, n-1$) such that $w = A_1A_2\ldots A_n$. In other words, the path representation of an element $w$ is the ordered product of the largest paths of the reduced form of $w$. It is easy to see that such decomposition always exists and is unique (this is a consequence of the uniqueness of the reduced form of a word). Furthermore, the image via $\pi$ of an element $w \in F$ whose path representation consists of more than one path is the empty map. A word $w$ of $F$ will be called a \textbf{loop} if $\pi(w)$ is an identity map of $\Gamma$. 

 We define $N_\Gamma$ as the normal subgroup of $F$ generated by the set of loops, i.e., $N_\Gamma$ is generated by the set:
 \[
    \{ zwz^{-1} : z \in F \text{ and }\pi(w) \text{ is an identity morphism of } \Gamma \}.
 \]

\begin{proposition} \label{p: presentation of universal groupoid}
    The map $\overline{\lfloor \gamma \rfloor} \in F/N_{\Gamma} \to \lceil \gamma \rceil \in G_\Gamma$ is an isomorphism of groups.
\end{proposition}
\begin{proof}
    Consider the group morphism $f: F \to G_{\Gamma}$ determined by the function 
    \[
       \gamma \in \Gamma \mapsto \lceil \gamma \rceil \in G_{\Gamma}
    \]
    
    Let $\lfloor \gamma_{1} \rfloor \lfloor \gamma_{2} \rfloor \ldots \lfloor \gamma_{n} \rfloor$ be a loop of $F$, then
    \[
        f\big( \lfloor \gamma_{1} \rfloor \lfloor \gamma_{2} \rfloor \ldots \lfloor \gamma_{n} \rfloor \big) = \lceil \gamma_{1} \rceil \lceil \gamma_{2} \rceil \lceil \gamma_{n} \rceil = \lceil \gamma_{1} \ldots \gamma_{n} \rceil = \lceil 1_{\operatorname{s}(\gamma_{n})} \rceil = 1_{G_{\Gamma}}.
    \]
    Therefore, we obtain a surjective morphism of groups $\overline{f}: F/N_{\Gamma} \to G_{\Gamma}$ such that $\overline{f}(\overline{\lfloor \gamma \rfloor}) = \lceil \gamma \rceil$ for all $\gamma \in \Gamma$. Conversely, the map $f': G_\Gamma \to F/N_\Gamma$ defined by $f'(\lceil \gamma \rceil) := \overline{\lfloor \gamma \rfloor}$ is a well-defined morphism of groups since $\lfloor \gamma \rfloor \lfloor \rho \rfloor \lfloor \gamma \rho \rfloor^{-1}$ forms a loop if $\operatorname{s}(\gamma) = \operatorname{t}(\rho)$. Furthermore, it is clear that $f'$ and $\overline{f}$ are mutually inverses.
\end{proof}

In view of Proposition~\ref{p: presentation of universal groupoid} we will identify $G_{\Gamma}$ with $F/ N_\Gamma$. A \textbf{minimal element} of $N_\Gamma$ is a basic generator $z \dagger w \dagger z^{-1}$ of $N_\Gamma$ such that the path representation of $z$ has no loops, i.e., the only loop in $zwz^{-1}$ is $w$.

 \begin{lemma} \label{l: minimal decomposition}
    The normal subgroup $N_\Gamma$ is generated by the set of minimal elements of $N_\Gamma$.
 \end{lemma}

 \begin{proof}
    It is enough to show that each basic generator of $N_\Gamma$ is a product of minimal elements. Let $z w z^{-1}$ be a basic element of $N_\Gamma$ and let $z = A_1 \ldots A_n$ be the path representation of $z$. Take $r:= \min\left\{ 1 \leq i \leq n : A_i \text{ is a loop} \right\}$, if such $r$ does not exist we are done. Otherwise, due to the minimality of $r$ the product of paths $A_1 \ldots A_{r-1} \mathbf{A_r} A_r^{-1} \dots A_1^{-1}$ is a minimal element of $N_\Gamma$. Therefore, we have that
    \[
      z w  z^{-1} =  (A_1 \dots A_{r-1} \mathbf{A_r} A_{r-1}^{-1} \dots A_1^{-1}) z_0 w z_0^{-1} (A_1 \dots A_{r-1} \mathbf{A_r}^{-1} A_{r-1}^{-1} \dots A_1^{-1}),
    \]
    where $z_0 = A_1 \dots A_{r-1} A_{r+1} \ldots A_n$. Thus, we can repeat this algorithm with $z_0 w z_0^{-1}$ until we obtain the desired decomposition.
 \end{proof}

 \begin{definition}
  A \textbf{minimal representation} of $w \in N_\Gamma$ is a tuple $(\xi_0, \xi_1, \ldots, \xi_n)$ of minimal elements of $N_\Gamma$ such that $w=\xi_0 \xi_1 \dots \xi_n.$ We say that a minimal representation $(\xi_i)_{i=0}^n$ of $w$ is \textbf{reduced} if
   \begin{enumerate}[(i)]
     \item $\xi_i \ldots \xi_j \neq 1_F$ for all $0 \leq i \leq j \leq n$;
     \item if $\xi_i = zuz^{-1}$ and $\xi_{i+1}= z v z^{-1}$, where $u$ and $v$ are loops, then $u \dagger v$.
   \end{enumerate}
 	\label{d: minimal representation}
 \end{definition}
The first condition is clear since we can erase the cancelations from the minimal representation, and it still will be a minimal representation. The second condition is because if $\underline{uv}$, then $\xi_i \xi_{i+1}= z \underline{uv} z^{-1}$ is a minimal element of $N_\Gamma$ since $\underline{uv}$ is a loop. Hence, $(\xi_0, \dots, \xi_i \xi_{i+1}, \dots, \xi_n)$ is a shorter minimal representation of $w$. Furthermore, we can set $\mathcal{R}$ as the function that maps no-reduced minimal representations to reduced minimal representations.

For a word $w \in F$ with path representation $ w = A_1 \ldots A_n$ we define $\mathcal{L}(w)=A_m$, where $m = \max \left\{ i : A_i \text{ is a loop} \right\}$, i.e., if $\mathcal{L}(w)$ there exists, then it is the last (from left to right) loop in the path representation of $w$.

 An easy observation about paths and loops is that the product of a no-loop path and a loop cannot be the identity. Moreover, if $u$ is a loop, $B$ is a no-loop path and $\pi(uB) \neq \emptyset$, then by Lemma \ref{l: product of paths target and source} we have that $\operatorname{T}(uB)=\operatorname{T}(u)$ and $\operatorname{S(uB)}=\operatorname{S}(B)$. 
 Now consider two minimal elements $\xi= zuz^{-1}$ and $\phi= yvy^{-1}$ such that $(\xi, \phi)$ is reduced. We classify the pair $(\xi, \phi)$ in two possible cases with respect to $\mathcal{L}(\xi \phi)$:
 \begin{itemize}
   \item We call $(\xi, \phi)$ a \textbf{P}-pair if $\mathcal{L}(\xi \phi) = \mathcal{L}(\phi)=v$.    
   \item We say that $(\xi, \phi)$ is a \textbf{C}-pair when $\mathcal{L}(\xi \phi) \neq v$.
 \end{itemize}
 Let $z = A_1 \dots A_n$ and $y = B_1 \ldots B_m$ be path representations. Then, the product $\xi \phi$ has the following behaviors:
 \begin{enumerate}[(i)]
   \item The first case happens when there is \textbf{no complete cancelation} of $z^{-1}$ or $y$ in the product $\xi \phi$. That is, there exists some $k \leq \min(n,m)$ such that
   \[
      \xi \phi = z u \dagger (A_n^{-1} \dots A_k^{-1} B_k \dots B_m) \dagger v y \text{ is reduced and } A_{k}^{^{-1}} B_{k} \neq 1_{F}.
   \]
   In the case where $A_k^{-1} \dagger B_k$, both of the above unlinks are justified. If $\underline{A_k^{-1} B_k}$, we must be careful when $k=n$ or $k=m$, but in both cases, the unlinks are justified by Lemma~\ref{l: product of paths target and source} and the fact that $u \dagger A_n^{-1}$ and $B_m \dagger v$. Thus, we have a \textbf{P}-pair.

   \item If $z = y$, then
   \[
      \xi \phi = z u \dagger v z^{-1}.
   \]
   Because, $(\xi, \phi)$ is reduced. In this case we have a \textbf{P}-pair.
   \item If $n < m$ and $A_i = B_i$ for $1 \leq i \leq n$, then
   \[
      \xi \phi = z \dagger (u B_{n+1} \dots B_m) \dagger v y^{-1}.
   \]
   Notice that may happen that $\underline{u B_{n+1}}$, in this case by Lemma~\ref{l: product of paths target and source} we have that $\operatorname{T}(u B_{n+1}) = \operatorname{T}(u)$, what justifies the left unlink of paths, the right unlink of paths is clear if $m > n+1$, and for the case $m = n+1$ we guarantee using again Lemma~\ref{l: product of paths target and source}. Thus, this case corresponds to a \textbf{P}-pair.
   \item If $n > m$ and $A_i = B_i$ for $1 \leq i \leq m$ and $A_{m+1} \dagger v$, then
   \[
      \xi \phi = z u \dagger A_n^{-1} \dots A_{m+1}^{-1} \dagger v y. 
   \]
   This case is also a \textbf{P}-pair.
   \item If $n > m$ and $A_i = B_i$ for $1 \leq i \leq m$ and $\underline{A_{m+1} v}$, then
   \begin{equation} \label{eq: C-pair}
      \xi \phi = z u \dagger A_n^{-1} \dots \underline{A_{m+1}^{-1} v} \dagger y. 
   \end{equation}
   The unlinks are guaranteed by Lemma~\ref{l: product of paths target and source}. This is the only case that corresponds to a \textbf{C}-pair.
 \end{enumerate}

 From the above case, we conclude that a \textbf{P}-pair happens exactly when 
 \begin{equation}
      \xi \phi = z \dagger u (z^{-1} y) \dagger v \dagger y^{-1}.
 	\label{eq: P-pair}
 \end{equation}
 Note that $u(z^{-1}y) \neq 1_F$. Moreover, if $(\xi, \phi)$ is a \textbf{C}-pair then $\mathcal{L}(\xi \phi) = \mathcal{L}(\xi) = u$.
\begin{definition}
   Let $w \in N_\Gamma$ a no-trivial element. A reduced minimal representation $(\xi_i)_{i=0}^n$ of $w$ satisfies the \textbf{P}-property if $(\xi_i, \xi_{i+1})$ is a \textbf{P}-pair for all $0 \leq i \leq n-1$. 
\end{definition}

\begin{lemma} \label{l: P-property implies loops}
    Let $w \in N_\Gamma$ a no-trivial word and $(\xi_i)_{i=0}^n$ a reduced minimal representation of $w$ with the \textbf{P}-property. Then, $w$ has loops in its reduced representation, in particular $\mathcal{L}(w) = \mathcal{L}(\xi_n)$. Moreover, if $\pi(w)  \neq \emptyset$, then $w$ is a loop.
\end{lemma}

\begin{proof}
   Set $\xi_i = z_i u_i z_i^{-1}$, where $u_i$ is a loop and each $z_i$ has no loops, then by \eqref{eq: P-pair} we obtain
   \[
      \xi_1 \ldots \xi_n = z_1 \dagger u_1 (z_1^{-1}z_2)\dagger u_2 (z_2^{-1}z_3) \dagger u_3(z_3^{-1}z_4) \ldots (z_{n-1}^{-1}z_n)\dagger u_n\dagger z_n.
   \]
    Thus, if $z_n = A_1 A_2 \ldots A_m$ is the block decomposition of $z_n$, then
    \[
       w = x \dagger u_n \dagger A_1 \dagger \ldots \dagger A_n,
    \]
    where $x = z_1 \dagger u_1 (z_1^{-1}z_2)\dagger u_2 (z_2^{-1}z_3) \dagger u_3(z_3^{-1}z_4) \ldots (z_{n-1}^{-1}z_n)$. By construction, $z_n$ has no loops. Thus, $A_i$ is not a loop for all $i$. Therefore, the first loop that appears in the block decomposition of $w$ (when reading from right to left) is $u_n$. Consequently, $\mathcal{L}(w) = u_n = \mathcal{L}(\xi_n)$.     

   For the last assertion, observe that if $z_1 \neq 1_F$, then by the unlink on the left side of $u_1$ we have that $\pi(w)= \emptyset$, thus $z_1 = 1_F$. Analogously, if $z_2 \neq 1_F$ the unlink of $u_2$ will imply that $\pi(w)= \emptyset$. Thus, by induction argument using the unlinks at the left of each $u_i$ one shows that $z_i = 1_F$ for all $1 \leq i \leq n$. Thus, $w = u_1 \ldots u_n$, since $\pi(w) \neq \emptyset$, we have that those loops cannot be unlinked, consequently $w = \underline{u_1 u_2 \ldots u_n}$ is a loop.  
\end{proof}

\begin{lemma} \label{l: basic switch}
   Let $\xi$ and $\phi$ be such that $(\xi, \phi)$ is a reduced representation of a word $w \in N_\Gamma$. If $(\xi, \phi)$ is a \textbf{C}-pair, then there exists a minimal element $\xi'$ such that $w = \xi \phi = \phi \xi'$, and $(\phi, \xi')$ is a minimal representation of $w$ and $(\phi, \xi')$ is a \textbf{P}-pair. 
\end{lemma}

\begin{proof}
    Let $\xi = zuz^{-1}$ and $\phi= yvy$. Since $(\xi, \phi)$ is a \textbf{C}-pair then it is of the form \eqref{eq: C-pair} and $\mathcal{L}(w)=u$, thus the path representations of $z$ and $y$ are $z = A_1 \ldots A_n$ and $y = A_1 \ldots A_m$, where $m < n$. Then, 
   \begin{align*}
      \xi \phi &= A_1 \dots A_n u A_n^{-1} \ldots A_{m+2}^{-1} \underline{A_{m+1}^{-1} v} \dagger A_m ^{-1} \ldots A_1^{-1} \\ 
      &= (A_1 \dots A_n) K (K^{-1} u K)  
   \end{align*}
   where $K = A_n^{-1} \ldots A_{m+2}^{-1} \underline{A_{m+1}^{-1} v} \dagger A_m ^{-1} \ldots A_1^{-1}$. Observe that
   \begin{align*}
      (A_1 \dots A_n) K &= (A_1 \dots A_n) A_n^{-1} \ldots A_{m+2}^{-1} \underline{A_{m+1}^{-1} v} \dagger A_m^{-1} \ldots A_1^{-1} \\ 
      &= A_1 \dots A_m \dagger v \dagger A_m^{-1} \ldots A_1^{-1} \\ 
      &= yvy = \phi.
   \end{align*}
    Notice that $K^{-1} \dagger u \dagger K$ is a minimal element of $N_\Gamma$. We set $\xi' = K^{-1}uK$. Thus, $w = \xi \phi = \phi \xi'$, $(\phi, \xi')$ is a reduced minimal representation of $w$ that also is a \textbf{P}-pair since $\mathcal{L}(\phi \xi')= \mathcal{L}(w) = u = \mathcal{L}(\xi')$.
\end{proof}

In view of Lemma~\ref{l: basic switch} we define the switch of a \textbf{C}-pair $(\xi, \phi)$ as $\mathcal{S}(\xi, \phi) := (\phi, \xi')$. Furthermore, if $(\xi_0, \dots, \xi_n)$ is a reduced minimal decomposition such that $(\xi_{i-1}, \xi_i)$ is a \textbf{C}-pair, we define
\begin{equation}
   \mathcal{S}^i(\xi_0, \dots, \xi_n) := (\xi_0, \dots, \xi_{i-2}, \xi_i, \xi_{i-1}', \xi_{i+1}, \dots, \xi_n),
\end{equation}
where $\mathcal{S}(\xi_{i-1}, \xi_i)= (\xi_i,, \xi_{i-1}')$.

\begin{proposition} \label{p: every normal word has P-representation}
   There exists an algorithm $\mathcal{P}$ such that for all word $w \in N_\Gamma \setminus \{ 1 \}$ and for all reduced minimal representation $(\xi_i)_{i=0}^n$ of $w$ we have that 
   \[
      \mathcal{P}(\xi_0, \dots, \xi_n) = (\phi_1, \dots, \phi_m), \, \, m \leq n,
   \]
   is a reduced minimal representation of $w$ with the \textbf{P}-property. 
\end{proposition}

\begin{proof}
    For a minimal representation, its length is defined as the number of elements in the tuple. We will use induction with respect to the length of the reduced minimal representation. 
    Note that if the length of the reduced minimal representations is $1$, then the result is trivial. Moreover, the case where the representation has length $2$ has already been proven in Lemma~\ref{l: basic switch}, providing the base for the induction.
    Assume that we have an algorithm $\mathcal{P}_i$ for all reduced minimal representations of length $i \leq n$. Let $\mathcal{A}_0=(\xi, \xi_1, \ldots, \xi_n)$ be a reduced minimal representation (of length $n+1$) of a no-trivial word $w \in N_\Gamma$. Recall that we have a reduction algorithm $\mathcal{R}$ that maps minimal representations to reduced minimal representations. Since $\mathcal{A}_0$ is a reduced minimal representation, then $(\xi_1, \dots, \xi_n)$ is a reduced minimal representation of $\xi_1 \xi_2 \ldots \xi_n$ of length $n$. Thus, $(\phi_1, \dots, \phi_m) :=\mathcal{P}_n(\xi_1, \dots, \xi_n)$ is a reduced minimal representation with the \textbf{P}-property and $m \leq n$. Therefore, $\overline{\mathcal{K}}_0:=(\xi,\phi_1, \dots, \phi_m )$ is a minimal representation of $w$. Then, $\mathcal{K}_0:= \mathcal{R}(\overline{\mathcal{K}}_0)$ is a reduced minimal representation of $w$. If the length of $\mathcal{K}_0$ is less than $n+1$, then we are done by the induction hypothesis. On the other hand, if the length of $\mathcal{K}_0$ is $n+1$ then $m = n$ and $\mathcal{K}_0 = \overline{\mathcal{K}}_0$. Thus, $(\xi, \phi_1, \dots, \phi_n)$ is a reduced minimal representation almost with the \textbf{P}-property, all the consecutive pairs are \textbf{P}-pairs, with the possible exception of $(\xi, \phi_1)$. If $\mathcal{K}_0$ has the \textbf{P}-property we are done, set $\mathcal{K}:= \mathcal{S}_{n}\mathcal{S}_{n-1} \ldots \mathcal{S}_{2}\mathcal{S}_{1}(\mathcal{K}_0)$. That is, we have the following chain:
   \begin{align*}
      \mathcal{K}_0=(\xi, \phi_1, \dots, \phi_n) 
      &\overset{\mathcal{S}^1}{\longmapsto} (\phi_1, \xi^{(1)}, \phi_2, \dots, \phi_n) \\ 
      &\overset{\mathcal{S}^2}{\longmapsto} (\phi_1, \phi_2, \xi^{(2)}, \dots, \phi_n) \\  
      & \quad \vdots \\
      &\overset{\mathcal{S}^i}{\longmapsto} (\phi_1, \dots,\phi_i, \xi^{(i)}, \phi_{i+1},\dots, \phi_n) \\  
      & \quad \vdots \\ 
      &\overset{\mathcal{S}^n}{\longmapsto} \mathcal{K}:=(\phi_1, \phi_2,  \dots, \phi_n, \xi^{(n)}),  
   \end{align*}
    where $\mathcal{S}(\xi, \phi_1) = (\phi_1, \xi^{(1)})$ and in general $\mathcal{S}(\xi^{(i)}, \phi_{i+1}) = (\phi_{i+1}, \xi^{(i+1})$ for all $1 \leq i \leq n-1$. Note that $\mathcal{K}$ is a minimal representation of $w$, and by construction, it possesses the \textbf{P}-property. Thus, if $\mathcal{K}$ is reduced, we are done. Conversely, if $\mathcal{K}$ is not reduced, then the length of $\mathcal{R}(\mathcal{K})$ is less than or equal to $n$. Therefore, the induction hypothesis leads to the desired conclusion.
\end{proof}

\begin{proposition} \label{p: only one path by class}
 Let $\alpha, \beta \in \Gamma$, such that neither $\alpha$ nor $\beta$ is an identity. Then, $\alpha = \beta$ if, and only if, $\overline{\lfloor \alpha \rfloor} = \overline{\lfloor  \beta \rfloor}$ in $G_\Gamma$. 
\end{proposition}

\begin{proof}
    Suppose that $\overline{\lfloor \alpha \rfloor} = \overline{\lfloor \beta \rfloor}$. Then, $\lfloor \alpha \rfloor \lfloor \beta \rfloor^{-1} \in N_\Gamma$. Thus, by Proposition~\ref{p: every normal word has P-representation} and Lemma~\ref{l: P-property implies loops} we have that $\lfloor \alpha \rfloor \lfloor \beta \rfloor^{-1}$ is $1_F$ or has a loop in its path representation. On one hand, if $\lfloor \alpha \rfloor \lfloor \beta \rfloor^{-1} = 1_F$ we are done. On the other hand, if $\lfloor \alpha \rfloor \lfloor \beta \rfloor^{-1} \neq 1_F$ has loops, the only possibility is that $\lfloor \alpha \rfloor \lfloor \beta \rfloor^{-1}$ is a loop (since $\lfloor \alpha \rfloor \lfloor \beta \rfloor^{-1}$ is already in its reduced representation form and neither $\alpha$ nor $\beta$ are identities). Therefore, $\alpha = \beta$.
\end{proof}

\begin{lemma}
   Let $g \in G_{\Gamma},$ if $g \neq 1$, then there exists unique $\gamma_0, \gamma_1, \dots, \gamma_n \in \Gamma$ such that $\gamma_i \gamma_{i+1} = \emptyset$ for all $i \in \{ 0, \dots, n-1 \}$, $\gamma_i$ is not an identity element of $\Gamma$ and $g = \overline{\lfloor \gamma_0 \rfloor} \overline{\lfloor \gamma_1 \rfloor} \ldots \overline{\lfloor \gamma_n \rfloor}$.
   \label{l: unique representations of elements in G Gamma}
\end{lemma}
\begin{proof}
 Let $w \in F$ such that $\overline{w} = g$, and $A_1 A_2 \dots A_m$ a reduced path representation of $w$. Observe that if $A_i$ is a loop, then $\overline{A_1 \dots A_{i-1} A_{i+1} \dots A_n} = g$. Thus, we can assume that the path representation of $w$ has no loops. Define $\gamma_i := \pi(A_i)$, therefore $\overline{ w' } = g$, where $w' = \lfloor \gamma_0 \rfloor \dots \lfloor \gamma_n \rfloor$. Observe that by construction $\{ \gamma_i \}$ satisfies the desired properties. For the uniqueness suppose that $g = \overline{\lfloor \rho_0 \rfloor} \overline{\lfloor \rho_1 \rfloor} \ldots \overline{\lfloor \gamma_n \rfloor}$, then
 \[
   \lfloor \gamma_0 \rfloor \dots \lfloor \gamma_n \rfloor \lfloor \rho_m \rfloor^{-1} \dots \lfloor \rho_0 \rfloor^{-1} \in N_\Gamma.
 \]
 Recall, that by Proposition~\ref{p: every normal word has P-representation} all the no-trivial elements of $N_\Gamma$ has loops in its path decomposition. By hypothesis $\lfloor \gamma_0 \rfloor \dots \lfloor \gamma_n \rfloor$ and $\lfloor \rho_m \rfloor^{-1} \dots \lfloor \rho_0 \rfloor^{-1}$ has no loops, if $\lfloor \gamma_n \rfloor \lfloor \rho_m \rfloor^{-1} \neq 1_F$, then the above product is a no-trivial element of $N_\Gamma$ with no loops, which is a contradiction. Thus, $\lfloor \gamma_n \rfloor \lfloor \rho_m \rfloor^{-1} = 1_F$. Hence, by an induction argument we conclude that $n=m$ and $\lfloor \gamma_i \rfloor = \lfloor \rho_i \rfloor$.
\end{proof}

By Proposition~\ref{p: presentation of universal groupoid}, Proposition~\ref{p: only one path by class} and Lemma~\ref{l: unique representations of elements in G Gamma} we obtain the following proposition:
\begin{proposition} \label{p: universal partial action group properties}
    Let $\Gamma \rightrightarrows X$ be a groupoid. Then, the group $G_{\Gamma}$ satisfies the following properties:
    \begin{enumerate}[(i)]
        \item for all $\alpha, \beta \in \Gamma \setminus \{ 1_x : x \in X \}$, $\lceil \alpha \rceil = \lceil b \rceil$ if, and only if, $\alpha = \beta$;
        \item for all $z \in G_{\Gamma} \setminus \{ 1 \}$, there exists unique $\gamma_0, \ldots, \gamma_n \in \Gamma \setminus \{ 1_x : x \in X \}$, such that
            \[
               z = \lceil \gamma_0 \rceil \lceil \gamma_1 \rceil \ldots \lceil \gamma_n \rceil,
            \]
            and $\operatorname{s}(\gamma_i) \neq \operatorname{t}(\gamma_{i+1})$ for all $i = 0, 1, \ldots, n-1$;
    \end{enumerate}
\end{proposition}

\subsection{The universal partial group action of a groupoid}
Now we can define a partial action of $G_\Gamma$ on $X$
\begin{equation}
    \theta^{\Gamma} := \Big( G^{\Gamma}, X, \{ X^{\Gamma}_{g} \}_{g \in G_{\Gamma}}, \{ \theta_{g}^{\Gamma} \}_{g \in G_{\Gamma}} \Big)
\end{equation}
such that:
\[
    X_{\omega}^{\Gamma}:= 
    \left\{\begin{matrix}
        X  & \text{ if } \omega = 1,\\ 
        \{ \operatorname{t}(\gamma) \}  & \text{ if } \omega = \lceil \gamma \rceil \text{ for some } \gamma \in \Gamma,\\ 
        \varnothing  &  \text{ otherwise}.
    \end{matrix}\right.
\]
For the morphisms we set
\begin{enumerate}[(i)]
    \item $\theta^{\Gamma}_1 := 1_X$, 
    \item $\theta^{\Gamma}_\omega := \emptyset$ if $\omega \neq \lceil \gamma \rceil$ for all $\gamma \in \Gamma$, and 
    \item $\theta_{\lceil \gamma \rceil}(\operatorname{s}(\gamma)) = \operatorname{t}(\gamma)$ for all $\gamma \in \Gamma$.
\end{enumerate}
Observe that the sets $\{ X_{\omega} \}$ and the maps $\{ \theta^{\Gamma}_{\omega} \}$ are well-defined by Proposition~\ref{p: universal partial action group properties}. Furthermore, by direct computations one verifies that $\theta^{\Gamma}$ is a partial group action of $G_{\Gamma}$ on $X$.

\begin{remark} \label{r: morphism of Psi theta gamma}
    Note that by definition, the set of morphism of $\Psi(\theta^{\Gamma})$ is
    \begin{align*}
        \{ (\theta_{w}(x), w, x) : x \in X_{w^{-1}} \}
        &= \{ (\theta_{\lceil \gamma \rceil}(x), \lceil \gamma \rceil, x) : x \in X_{\lceil \gamma^{-1} \rceil}  \} \\
        &= \{ (\operatorname{t}(\gamma), \lceil \gamma \rceil, \operatorname{s}(\gamma) ) : \gamma \in \Gamma \setminus \{ 1_x \}_{x \in X} \} \sqcup \{ (x, 1_x, x) \}_{x \in X}.
    \end{align*}
\end{remark}

\begin{proposition} \label{p: every groupoid is a partial action groupoid}
  The partial action groupoid $\Psi(\theta^\Gamma)$ is isomorphic to the groupoid $\Gamma$. Therefore, every groupoid is a partial action groupoid.
\end{proposition}
\begin{proof}
   We define the maps $\eta_\Gamma: \Psi(\theta^\Gamma) \to \Gamma$ such that
   \begin{equation}
        \eta_\Gamma(y, g, x) = \begin{cases}
        \gamma, &\text{ if } g = \lceil \gamma \rceil \text{ for some } \gamma \in \Gamma \setminus \{ 1_x \}_{x \in X}\\
        1_x, &\text{ if }g = 1_{G_{\Gamma}}\\
        
      \end{cases}
   \end{equation}
   and
   \begin{equation}
      \eta_\Gamma^{-1}: \Gamma \to \Psi(\theta^\Gamma)  \text{ such that } \eta^{-1}_\Gamma(\gamma)= (\operatorname{t}(\gamma), \lceil \gamma \rceil, \operatorname{s}(\gamma)). 
   \end{equation}
   Notice that $\eta_\Gamma$ is well-defined by Remark~\ref{r: morphism of Psi theta gamma}. Furthermore, it is clear that $\eta_\Gamma^{-1}$ is well-defined since $X^\Gamma_{\lceil \gamma \rceil} = \left\{ \operatorname{t}(\gamma) \right\}$. Finally, by direct computations one verify that $\eta_\Gamma$ and $\eta_\Gamma^{-1}$ are mutually inverses, and that $\eta_\Gamma$ is a morphism of groupoids.
\end{proof}

Let $f: \Gamma {\rightrightarrows} X_{\Gamma} \to \Gamma' {\rightrightarrows} X_{\Gamma'}$ be a morphism of groupoids. We can see $f$ as a pair $(f_0, f_1)$ such that $f_0: X_\Gamma \to X_{\Gamma'}$ is a function, and $f_1 : \Gamma \to \Gamma'$ is the respective function between morphism of the groupoids. Thus, $f_1$ induces a morphism of groups $\tilde{\varphi}: F_\Gamma \to G_{\Gamma'}$ (where $F_{\Gamma}$ is free group determined by $\Gamma$) such that $\tilde{\varphi}(\lfloor \gamma \rfloor) = \overline{\lfloor f_1(\gamma) \rfloor}$. Since $f_1$ sends identity elements to identity elements we conclude that if $\lfloor \gamma_0 \rfloor \lfloor \gamma_1 \rfloor \dots \lfloor \gamma_n \rfloor$ is a loop, then $\lfloor f_1(\gamma_0) \rfloor \lfloor f_1(\gamma_1) \rfloor \dots \lfloor f_1(\gamma_n) \rfloor$ is a loop, then $N_{\Gamma} \subseteq \ker \tilde{\varphi}$. Thus, there exists a unique homomorphism of groups $\varphi_{f}: G_{\Gamma} \to G_{\Gamma'}$ such that $\varphi_{f}(\overline{\lfloor \gamma \rfloor})=\overline{\lfloor f_1(\gamma) \rfloor}$. Define $\Phi(\Gamma):= \theta^{\Gamma}$ and $\Phi(\Gamma'):= \theta^{\Gamma'}$, then we obtain a morphism of partial actions $\Phi(f): \Phi(\Gamma) \to \Phi(\Gamma')$ such that
\begin{equation}
    \Phi(f) := (f_0, \varphi_{f}):  \theta^\Gamma \to \theta^{\Gamma'}.
\end{equation}

Hence, we can state the following proposition which corresponds to the first part of Theorem~\ref{t: groupoids are partial actions}:

\begin{proposition}
    $\Phi: \textbf{Grpd} \to \textbf{PA}$ such that $\Phi(\Gamma) = \theta^{\Gamma}$ and $\Phi(f)= (f_0, \varphi_{f})$ is a covariant functor such that $\Psi \circ \Phi \cong 1_{\textbf{Grpd}}$.
\end{proposition}
\begin{proof}
    The fact that $\Phi$ is a functor was done in the above discussion. For the last part it is enough to observe that the map $\eta$ defined in the proof of Theorem~\ref{p: every groupoid is a partial action groupoid} defines a natural isomorphism between $\Psi \circ \Phi$ and $1_{\textbf{Grpd}}$.
\end{proof}

\subsection{The universal property of the groupoid partial action}

Our objective now is show that the partial action $\theta^\Gamma=( G_\Gamma, X, \left\{ X^\Gamma_g \right\}, \left\{ \theta^\Gamma_g \right\})$ can be characterized (up to isomorphism) by certain universal property. 

\begin{lemma}
   Let $\alpha$ be a partial action of $S$ on $Y$ and $f: \Gamma \to \Psi(\alpha)$ an isomorphism of groupoids. Then, there exists $f_0: X \to Y$ bijection of set, and $f_1: \Gamma \to S$ morphisms of groupoids, such that $f(\gamma)=\big(f_0(\operatorname{t}(\gamma)), f_1(\gamma), f_0(\operatorname{s(\gamma)}) \big)$.
   \label{l: decomposition of isomorphism of partial actions groupoids}
\end{lemma}

\begin{proof}
   Observe that if $\gamma \in \Gamma$, then $f(\gamma) = (f_0^{\operatorname{t}}(\gamma), f_1(\gamma), f_0^{\operatorname{s}}(\gamma))$ for some maps $f_1, f_0^{\operatorname{s}}, f_0^{\operatorname{t}}$. Since, $f$ maps identities into identities and
   \[
      f(1_z) = (f_0^{\operatorname{t}}(1_z), f_1(1_z), f_0^{\operatorname{s}}(1_z)),
   \]
    we have that $f_0^{\operatorname{s}}(1_z) =f_0^{\operatorname{t}}(1_z)$ for all $z \in X$. Thus, we define $f_0(z) := f_0^{\operatorname{s}}(1_{z})$. Since $1_{\operatorname{t}(\gamma)} \circ \gamma = \gamma = \gamma \circ 1_{\operatorname{s(\gamma)}}$ and $f$ is a morphism is groupoids, we conclude that $f_0(\operatorname{s(\gamma)}) = f_0^{\operatorname{s}}(\gamma)$ and $f_0(\operatorname{t}(\gamma)) = f_0^{\operatorname{t}}(\gamma)$. Therefore, 
   \[
      f(\gamma)=(f_0(\operatorname{t}(\gamma)), f_1(\gamma), f_0(\operatorname{s(\gamma)})).
   \]
   Finally, $f_0$ is a bijection since $f$ is a bijection and $f_1$ is a morphism of groupoids since $f$ is a morphism of groupoids.
\end{proof}

Now we can prove the following proposition that corresponds to the second part or Theorem~\ref{t: groupoids are partial actions}:

\begin{proposition}
   Let $\theta = (S, Y, \{ Y_s \}, \{ \theta_h \})$ be a partial action such that $f: \Gamma \to \Psi(\theta)$ is an isomorphism of groupoids. Then, there exists a unique morphism of partial actions $\phi=(\phi_0, \phi_1) : \theta^\Gamma \to \theta$ such that $f \circ \eta_\Gamma = \Psi(\phi)$.
\[\begin{tikzcd}
	{\theta^\Gamma} & {\Psi(\theta^\Gamma)} \\
	\theta & \Gamma & {\Psi(\theta)}
	\arrow["{\exists ! \phi}"', from=1-1, to=2-1]
	\arrow["{\eta_\Gamma}"', from=1-2, to=2-2]
	\arrow["f", from=2-2, to=2-3]
	\arrow["{\Psi(\phi)}", from=1-2, to=2-3]
\end{tikzcd}\]
   \label{p: Universal property of G Gamma}
\end{proposition}
\begin{proof}
   If $f$ is isomorphism of groupoids, then by Lemma~\ref{l: decomposition of isomorphism of partial actions groupoids} we have that $f(\gamma) = (f_0(\operatorname{t}(\gamma)), f_1(\gamma), f_0(\operatorname{s}(\gamma)))$ for some bijection $f_0$ and morphism of groupoids $f_1$. Keeping in mind Remark~\ref{r: morphism of Psi theta gamma} observe that the composition $f \circ \eta_{\Gamma}$ is given by
   \begin{equation*}
       f \circ \eta_{\Gamma}(x, 1_{G_{\Gamma}}, x) = f(1_x) = (f_0(x), 1_{G_{\Gamma}}, f_0(x)),
   \end{equation*}
   and
   \begin{equation} \label{eq: composition proof P universal property}
       f \circ \eta_{\Gamma}(\operatorname{t}(\gamma), \lceil \gamma \rceil, \operatorname{s}(\gamma)) = f(\gamma) = (f_0(\operatorname{t}(\gamma)), f_1(\gamma), f_0(\operatorname{s}(\gamma))),
   \end{equation}
    Therefore, the morphism of partial actions if there exists must be $\phi = (f_0, \phi_1)$, such that $\phi_1(\lceil \gamma \rceil)= f_1(\gamma)$ for all $\gamma \in \Gamma$. Thus, we have to verify the exists of such a map $\phi_1$. Observe that the map $f_1: \Gamma \to S$ give rise to a morphism of groups $\tilde{\phi}_1: F \to S$, such that $\tilde{\phi}_1(\lfloor \gamma \rfloor) = f_1(\gamma).$ Let $\lfloor \gamma_0 \rfloor \lfloor \gamma_1 \rfloor \dots\lfloor \gamma_n \rfloor$ be a loop, therefore, $\gamma_0 \gamma_1 \dots \gamma_n = 1_x,$ for some $x \in X_\Gamma$. Thus,
   \[
       \tilde{\phi}_1(\lfloor \gamma_0 \rfloor \lfloor \gamma_1 \rfloor \dots\lfloor \gamma_n \rfloor) = f_1(\gamma_0)f_1(\gamma_1) \dots f_1(\gamma_n) = f_1(\gamma_0 \gamma_1 \dots \gamma_n)= f_1(1_x)=1_S.
   \]
   Then, $N_\Gamma \subseteq \ker \tilde{\phi_1}$. Hence, there exists a morphism of groups $\phi_1: G_\Gamma \to S$, such that $\phi_1(\lceil \gamma \rceil) = f_1(\gamma)$. We set $\phi = (f_0, \phi_1)$. Thus, Equation~\eqref{eq: composition proof P universal property} takes the form
   \[
       (y, \lceil \gamma \rceil,x) \longmapsto (f_0(y), \phi_1(\lceil \gamma \rceil), f_0(x)).
   \]
   Whence we conclude that $f_0(X_\omega) \subseteq Y_{\phi_1(\omega)}$ and $f_0(\theta^\Gamma_\omega(x)) = \theta_{\phi_1(\omega)}(f_0(x))$ for all $\omega \in G_{\Gamma}$, that is, $\phi$ is a morphism of partial actions.
\end{proof}

\begin{corollary}
In the conditions of Proposition~\ref{p: Universal property of G Gamma} we have that there exists a bijection $\phi_0: X \to Y$ and morphism of groups $\phi_1: G_\Gamma \to S$ such that the isomorphism $f \circ \eta_\Gamma: \Psi(\theta^{\Gamma}) \to \Psi(\alpha)$ is such that $(y, g, x) \mapsto (\phi_0(y), \phi_1(g), \phi_0(x))$.
   \label{c: Universal property of G Gamma}
\end{corollary}

From Proposition~\ref{p: Universal property of G Gamma} and a classical universal property argument we conclude that $\theta^\Gamma$ is the unique (up to isomorphism) partial action that satisfies the universal property of partial action groupoids.

\subsection{Adjunction}

Now we will proof the last part of Theorem~\ref{t: groupoids are partial actions}.

\begin{lemma} \label{l: epsilon theta map}
    Let $\theta$ be a partial action. Then, there exists a morphism of partial actions $\varepsilon_\theta: \Phi \Psi (\theta) \to \theta$, such that $\Psi(\varepsilon_\theta) = \eta_{\Psi(\theta)}$.
\end{lemma}
\begin{proof}
    Consider the identity map $1 : \Psi(\theta) \to \Psi(\theta)$, the by $(ii)$ of Theorem~\ref{t: groupoids are partial actions} there exists a morphism of partial actions $\varepsilon_\theta: \Phi \Psi(\theta) \to \theta$ such that the diagram
    \[
        \begin{tikzcd}
            {\Psi(\theta)} & {\Psi(\theta)} \\
            {\Psi\Phi\Psi(\theta)}
            \arrow["1", from=1-1, to=1-2]
            \arrow["{\Psi(\varepsilon_\theta)}"', curve={height=12pt}, dashed, from=2-1, to=1-2]
            \arrow["{\eta_{\Psi(\theta)}}", from=2-1, to=1-1]
        \end{tikzcd}
    \]
\end{proof}

\begin{lemma} \label{l: Psi one to one condition}
    Let $\Gamma \rightrightarrows X$ be a groupoid and $\theta$ a partial action of $G$ on a set $Z$. If $f, h \in \hom(\Phi(\Gamma), \theta)$ are two maps such that $\Psi(f) = \Psi(h)$, then $f = h$.
\end{lemma}
\begin{proof}
By Remark~\ref{r: morphism of Psi theta gamma} we know that
\begin{align*}
    \Psi \Phi(\Gamma) 
    = \{ (\operatorname{t}(\gamma), \lceil \gamma \rceil, \operatorname{s}(\gamma) ) : \gamma \in \Gamma \setminus \{ 1_x \}_{x \in X} \} \sqcup \{ (x, 1_{G_{\Gamma}}, x) \}_{x \in X},
\end{align*}
since $\Psi(f)=\Psi(h)$, then
\[
    \Big(f_0(\operatorname{t}(\gamma)), \, f_1(\lceil \gamma \rceil), \, f_0(\operatorname{s}(\gamma)) \Big) = \Psi(f)=\Psi(h) = \Big(h_0(\operatorname{t}(\gamma)), \, h_1(\lceil \gamma \rceil), \, h_0(\operatorname{s}(\gamma)) \Big),
\]
for all $\gamma \in \Gamma \setminus \{ 1_x \}_{x \in X}$, and 
\[
    \Big(f_0(x), \, 1, \, f_0(x) \Big) = \Psi(f)=\Psi(h) = \Big(h_0(x), \, 1, \, h_0(x) \Big),
\]
for all $x \in X$. Whence we conclude that $f_0(x) = h_0(x)$ for all $x \in X$ and $f_1(\lceil \gamma \rceil) = h_1(\lceil \gamma \rceil)$ for all $\gamma \in \Gamma$. Since $\{ \lceil \gamma \rceil : \gamma \in \Gamma \}$ generates $G_\Gamma$ as a group, then $f_1 = h_1$. Hence, $f = (f_0, f_1) = (h_0, h_1) = h$.
\end{proof}

Let $\Gamma \rightrightarrows X$ be a groupoid and $\theta$ a partial group action of a group $G$ on the set $Z$. Define $\lambda: \hom(\Phi(\Gamma), \theta) \to \hom(\Gamma, \Psi(\theta))$ such that $\lambda(f):= \Psi(f) \circ \eta^{-1}_{\Gamma}$ for all $f \in \hom(\Phi(\Gamma), \theta)$,
\[
    \lambda(f):=\Gamma 
    \overset{\eta^{-1}_{\Gamma}}{\to} \Psi \Phi(\Gamma) 
    \overset{\Psi(f)}{\to} \Psi(\theta),
\]
and $\lambda': \hom(\Gamma, \Psi(\theta)) \to \hom(\Phi(\Gamma), \theta)$ such that $\lambda'(h)= \varepsilon_{\theta} \circ \Phi(h)$ for all $h \in \hom(\Gamma, \Psi(\theta))$
\[
    \lambda'(h):= \Phi(\Gamma) \overset{\Phi(h)}{\to} \Phi \Psi(\theta) \overset{\varepsilon_{\theta}}{\to} \theta,
\]
where $\varepsilon_{\theta}$ is the map obtained in Lemma~\ref{l: epsilon theta map}. Let $h \in \hom(\Gamma, \Psi(\theta))$, then
\begin{align*}
    \lambda \lambda'(h)
    &= \lambda(\varepsilon_{\theta} \circ \Phi(h)) \\
    &= \Psi(\varepsilon_{\theta} \circ \Phi(h)) \circ \eta_{\Gamma}^{-1} \\
    &= \Psi(\varepsilon_{\theta}) \circ \Psi \Phi(h) \circ \eta_{\Gamma}^{-1} \\
    (\text{ by Lemma~\ref{l: epsilon theta map}}) &= \eta_{\Psi(\theta)} \circ \Psi \Phi(h) \circ \eta_{\Gamma}^{-1} \\
    &= h,
\end{align*}
the last equality holds due to the fact that $\eta: \Psi \Phi \to 1_{\textbf{Grpd}}$ is a natural isomorphism. Thus, $\lambda$ is a surjective map. Let $f,h \in \hom(\Phi(\Gamma), \theta)$ such that $\lambda(f)=\lambda(h)$, therefore $\Psi(f) \circ \eta_{\Gamma} = \Psi(h) \circ \eta_{\Gamma}$, thus $\Psi(f) = \Psi(h)$. Hence, by Lemma~\ref{l: Psi one to one condition} we conclude that $\lambda$ is injective. Thus, for all groupoid $\Gamma$ and partial action $\theta$ we obtain a bijection $\lambda_{\Gamma, \theta}: \hom(\Phi(\Gamma), \theta) \to \hom(\Gamma, \Psi(\theta))$. 
To verify that $\lambda$ establishes an adjunction between the functors $\Phi$ and $\Psi$, we need to confirm that the following diagram commutes
\[\begin{tikzcd}
	{\hom(\Phi(\Gamma), \theta)} && {\hom(\Gamma, \Psi(\theta))} \\
	{\hom(\Phi(\Gamma'), \theta')} && {\hom(\Gamma', \Psi(\theta'))}
	\arrow["{\lambda_{\Gamma, \theta}}", from=1-1, to=1-3]
	\arrow["{\lambda_{\Gamma', \theta'}}", from=2-1, to=2-3]
	\arrow["{\hom(f, \Psi(h))}", from=1-3, to=2-3]
	\arrow["{\hom(\Phi(f), h)}"', from=1-1, to=2-1]
\end{tikzcd}\]
for all partial actions $\theta$ and $\theta'$, all groupoids $\Gamma$ and $\Gamma'$, and any pair of morphisms $f: \Gamma' \to \Gamma$ and $h: \theta \to \theta'$. Let $\sigma: \Phi(\Gamma) \to \theta$ morphism of partial actions,  thus
\begin{align*}
   \lambda_{\Gamma', \theta'} \circ \hom(\Phi(f), h)(\sigma)
   &= \lambda_{\Gamma', \theta'}(h \circ \sigma \circ \Phi(f)) \\
   &= \Psi(h \circ \sigma \circ \Phi(f)) \circ \eta_{\Gamma'}^{-1} \\
   &= \Psi(h) \circ \Psi(\sigma) \circ \Psi\Phi(f) \circ \eta_{\Gamma'}^{-1} \\
   (\flat) &= \Psi(h) \circ \Psi(\sigma) \circ  \eta_{\Gamma}^{-1} \circ f \\
   &= \hom(f, \Phi(h)) \big( \Psi(\sigma) \circ \eta_{\Gamma}^{-1} \big) \\
   &= \hom(f, \Phi(h)) \circ \lambda_{\Gamma, \theta} (\sigma),
\end{align*}
equality $(\flat)$ holds since $\eta: \Psi \Phi \to 1_{\textbf{Grpd}}$ is a natural transformation, and therefore $\Psi \Phi(f) = \eta_{\Gamma}^{-1} \circ f \circ \eta_{\Gamma'}$. Thus, we have proved the last part of Theorem~\ref{t: groupoids are partial actions}.

\section{Partial actions associated to the same groupoid}

\underline{From now on}, we will identify $\Gamma {\rightrightarrows} X$ with $\Psi \Phi (\Gamma)$. Our objective now is study the class of partial actions $\alpha$ such that $\Psi(\alpha) \cong \Gamma$.

\begin{remark}
   Let $\alpha = (S, Y, \{ Y_h \}, \{ \alpha_h \})$ be a partial action of $S$ on $Y$ such that $\Psi(\alpha) \cong \Gamma$. Let $f: \Gamma \to \Psi(\alpha)$ be an isomorphism of groupoids, then by Corollary~\ref{c: Universal property of G Gamma} we have that $f=(\phi_0, \phi_1)$. We can define a new partial action $\tilde{\alpha}:= (S, X, \{ X_h \}, \{ \tilde{\alpha}_h \})$ of $S$ on $X$ isomorphic to $\alpha$. Indeed, define $X_h := \phi_0^{-1}(Y_h)$ and $\tilde{\alpha}_h := \phi_0^{-1} \alpha_h \phi_0$. It is easy to see that $(\phi_0^{-1}, 1_S): \alpha \to \tilde{\alpha}$ is an isomorphism of partial actions. Then, the isomorphism of their respective partial action groupoid is given by
   \[
      (y, h, x) \longmapsto (\phi_0^{-1}(y), h, \phi_0^{-1}(x)).
   \]
   Thus, the composition of isomorphisms
   \[
      \Gamma \longrightarrow \Psi(\alpha) \longrightarrow \Psi(\tilde{\alpha})
   \]
   is such that 
   \begin{equation}
      (y,g,x) \longmapsto (\phi_0(y), \phi_1(g), \phi_0(x)) \longmapsto (y, \phi_1(g), x).
      \label{eq: first restriction of partial actions}
   \end{equation}
    Therefore, we can restrict ourselves to the class of partial actions on $X$ such that the isomorphisms of their respective groupoid partial actions are of the form \eqref{eq: first restriction of partial actions}, i.e., a partial action of such form is a partial action $\alpha$ of a group $S$ on $X$ equipped with a group homomorphism $\phi: G_{\Gamma} \to S$ such that $(y, g, x) \in \Gamma \mapsto (y, \phi(g), x) \in \Psi(\alpha)$ is an isomorphism of groupoids.
   \label{r: partial action restriction to X}
\end{remark}
 
We can impose further restrictions, consider a partial action $\alpha = (S, X, \{ X_h \}, \{ \alpha_h \})$ of $S$ on $X$ of the form \eqref{eq: first restriction of partial actions} with morphism of groups $\phi: G_{\Gamma} \to S$. We can then define a new partial action $\hat{\alpha}=(\im \phi,\{ X_h \}_{h \in \im \phi}, \{ \alpha_h \}_{h \in \im \phi})$ of $\im \phi$ on $X$. It is clear that $\hat{\alpha}$ is a partial action, and this partial action contains all the information about $\alpha$ but the original group, as we will prove in the following lemma:
\begin{lemma}
   Let $\alpha$ be a partial action of $S$ on $X$ such that $\Psi(\alpha) \cong \Gamma$ and the isomorphism has the form    
   \[
      (y, g, x) \mapsto (y, \phi(g), x),
   \]
   where $\phi: G_\Gamma \to S$ is the morphism of groups given in Corollary~\ref{c: Universal property of G Gamma}. If $\alpha_h \neq \emptyset$, then $h \in \im \phi$.
   \label{l: depuration of a partial action}
\end{lemma}
\begin{proof}
   If $\alpha_h \neq \emptyset$, then there exists $x \in X_{h^{-1}}$, consequently $(\alpha_h(x), h, x) \in \Psi(\alpha)$. Since, $(1_X, \phi)$ is an isomorphism of groupoids, then $(\alpha_h(x), \phi(g), x) = (\alpha_h(x), h, x)$, for some $g \in G_{\Gamma}$. Hence, $\phi(g)=h$. 
\end{proof}

\begin{definition} \label{d: reduced partial action}
   Let $\theta = (G, X, \left\{ X_g \right\}_{g \in G}, \left\{ \theta \right\}_{g \in G})$ be a partial action of $G$ on $X$. We define $\operatorname{red} G$ as the subgroup of $G$ generated by the set $\{ g \in G : \theta_{g} \neq \emptyset \}$. Then, we obtain a new partial action of $\operatorname{red} G$ on $X$ given as follows:
   \begin{equation}
       \operatorname{red}\theta := \big( \operatorname{red}G, X, \left\{ X_g \right\}_{g \in \operatorname{red}G}, \left\{ \theta \right\}_{g \in \operatorname{red}G} \big).
   \end{equation}
   We say that $\theta$ is a \textbf{reduced partial action} if $\theta = \operatorname{red}\theta$, or equivalently, if $\operatorname{red} G = G$. 
\end{definition}

\begin{definition}
    Let $\alpha = (S, X, \left\{ X_h \right\}_{h \in G}, \left\{ \alpha_h \right\}_{h \in S})$ be a partial action of $S$ on $X$. We say that $\alpha$ is a \textbf{reduced partial action with groupoid $\Gamma$} if there exists a surjective homomorphism of groups $\phi : G_{\Gamma} \to S$ such that  $(y, g, x) \in \Psi(\alpha) \longmapsto (y, \phi(g), x) \in \Gamma$ is an isomorphism of groupoids.
\end{definition}

\begin{remark}
    Note that if $\alpha$ is a reduced partial action with groupoid $\Gamma$, then it is also a reduced partial action in the sense of Definition~\ref{d: reduced partial action}. Indeed, the group $G_{\Gamma}$ is generated by the set $\{ \lceil \gamma \rceil : \gamma \in \Gamma \}$, and $X^{\Gamma}_{\lceil \gamma \rceil} \neq \varnothing$ for all $\gamma \in \Gamma$. Therefore, $X_{\phi(\lceil \gamma \rceil)} \neq \varnothing$ since $X_{\lceil \gamma \rceil}^{\Gamma} \subseteq X_{\phi(\lceil \gamma \rceil)}$. Finally, $S$ is generated by $\{ \phi(\lceil \gamma \rceil) : \gamma \in \Gamma\}$ since $\phi$ is surjective.
\end{remark}

Since our objective is to study the partial actions with a partial action groupoid isomorphic to $\Gamma$, Remark~\ref{r: partial action restriction to X} and Lemma~\ref{l: depuration of a partial action} tell us that it is enough to study the reduced partial actions with groupoid $\Gamma$.

\begin{proposition}
   Let $\alpha = (S, X, \left\{ D_h \right\}_{h \in S}, \left\{ \alpha_h \right\}_{h \in S})$ be a reduced partial action with groupoid $\Gamma$, with respective surjective morphism of groups $\phi: G_{\Gamma} \to S$. Then,
   \begin{equation}
      D_{h} = \bigcup_{g \in \phi^{-1}(h)} X^\Gamma_g;
      \label{eq: induced partial action i}
   \end{equation}
   and
   \begin{equation}
      \alpha_h(x) = \theta^\Gamma_g(x),
      \label{eq: induced partial action ii}
   \end{equation}
   where $g \in \phi^{-1}(h)$ is the only pre-image of $h$ such that $x \in X_{g^{-1}}^{\Gamma}$. Hence, the partial action $\alpha$ is determined by $\phi.$
	\label{p: groupoid partial actions are determined by phi}
\end{proposition}
\begin{proof}
   This is a direct consequence of the fact that the map
   \[
      (y, g, x) \overset{(1_X, \phi)}{\longmapsto} (y, \phi(g), x)
   \]
   is an isomorphism of groupoids.
\end{proof}

Proposition~\ref{p: groupoid partial actions are determined by phi} tell us that to obtain reduced partial actions with groupoid $\Gamma$ we only have to look for surjective morphisms of groups such that \eqref{eq: induced partial action i} and \eqref{eq: induced partial action ii} determines a partial action. Henceforth, we are interested in the following set 
\begin{equation}
  \left\{ N \trianglelefteq G_\Gamma : \text{ the quotient homomorphism } G_\Gamma \to G_\Gamma/ N \text{ determines a partial action} \right\}.
\end{equation}

\begin{proposition}
   Let $\phi: G_\Gamma \to S$ be a surjective morphism of groups. For all $h \in S$, we define $\Gamma_h := \{ \gamma \in \Gamma : \phi(\lceil \gamma \rceil) = h\}$, and set
   \begin{enumerate}[(i)]
       \item $D_{h} := \{ \operatorname{t}(\gamma) : \gamma \in \Gamma_{h} \} = \{ \operatorname{s}(\rho) : \rho \in \Gamma_{h^{-1}} \}$,
       \item $\alpha_h(\operatorname{s}(\rho)) := \operatorname{t}(\rho) = \theta^{\Gamma}_{\lceil \rho \rceil}(\operatorname{s}(\rho))$, for all $\operatorname{s}(\rho) \in D_{h^{-1}}$.
   \end{enumerate}
   Then, $\alpha = (S, X, \left\{ D_h \right\}, \left\{ \alpha_h \right\})$ is a well-defined reduced partial action with groupoid $\Gamma$ if, and only if, for all $z \in \ker \phi$, $z \neq 1_{G_\Gamma}$, we have that $z \neq \lceil \gamma \rceil$ for all $\gamma \in \Gamma$.
   Note that $D_{h}$ may be empty and consequently $\alpha_{h}$ be the empty function.
   Furthermore, if $\alpha$ is well-defined then $\Psi(1_X, \phi): \Psi(\theta^{\Gamma}) \to \Psi(\alpha)$ is an isomorphism of groupoids.

   \label{p: characterization of ideal normal subgroups}
\end{proposition}
\begin{proof}
    Suppose that $\alpha$ is a well-defined partial action. Let $g \in G_{\Gamma}$, $g \neq 1_{G_\Gamma}$, such that $g = \lceil \gamma \rceil$ for some $\gamma \in \Gamma$. Since $g \neq 1_{G_{\Gamma}}$, then $\gamma$ is not an identity. Suppose that $\phi(g) = 1_{S}$; therefore, $\alpha_1(\operatorname{s}(\gamma)) = \operatorname{t}(\gamma) \neq \operatorname{s}(\gamma)$, which contradicts the fact that $\alpha$ is a well-defined partial action.     

    Conversely, suppose that for all $z \in \ker \phi \setminus \{ 1_{G_\Gamma}\}$, we have $z \neq \lceil \gamma \rceil$ for all $\gamma \in \Gamma$. First, we have to verify that $\alpha_h$ is well-defined, i.e., we have to prove that for all $x \in D_{h^{-1}}$, there exists a unique $\gamma \in \phi^{-1}(h)$ such that $\operatorname{s}(\gamma) = x$.
    If $\lceil \gamma \rceil, \lceil \rho \rceil \in \phi^{-1}(h)$, then $\lceil \gamma \rceil \lceil \rho^{-1} \rceil \in \ker \phi$. Hence, by hypothesis $\operatorname{s}(\gamma) \neq \operatorname{t}(\rho^{-1})$.
    Therefore, $\alpha_h$ is well-defined for all $h \in S$. Next, we have to verify that $\alpha$ is a partial action. For any pair $h, h' \in S$ and $x \in \dom \alpha_h \alpha_s$ we have
   \[
      \alpha_h \alpha_{h'}(x) = \theta^\Gamma_g \theta^\Gamma_{g'}(x) = \theta^\Gamma_{gg'}(x) = \alpha_{hh'}(x),
   \]
   where $g' \in \phi^{-1}(h')$, $x \in X_{(g')^{-1}}^\Gamma$, $g \in \phi^{-1}(h)$ and $\theta_{g'}(x) \in X_{g^{-1}}$.

    Finally, we have to show that
   \[
      (y, g, x) \in \Gamma \overset{(1_X, \phi)}{\longmapsto} (y, \phi(g), x) \in \Psi(\alpha)
   \]
    is an isomorphism of groupoids. By the definition of $\alpha$ it is clear that $(1_X, \phi)$ is a surjective morphism of groupoids. Moreover, we already know that for any $(y, h, x) \in \Psi(\alpha)$, there exists a unique $\lceil \gamma \rceil \in G_{\Gamma}$ such that $\lceil \gamma \rceil \in \phi^{-1}(h)$ and $\operatorname{s}(\gamma)= x$, this proves that the map $(1_{X}, \phi)$ is injective.
\end{proof}

By Proposition~\ref{p: groupoid partial actions are determined by phi} and Proposition~\ref{p: characterization of ideal normal subgroups} we have that the set 
\begin{equation}
   \mathcal{W}_\Gamma := \{ N \trianglelefteq  G_\Gamma : g \neq \lceil \gamma \rceil, \text{ for all } g \in N \setminus \{ 1_{G_\Gamma}\} \text{ and } \gamma \in \Gamma\}
   \label{eq: admisible normal subgroups}
\end{equation}
determines all the reduced partial actions with groupoid $\Gamma$.

Note that the partial action determined by $N \in \mathcal{W}_{\Gamma}$ defined in Proposition~\ref{p: characterization of ideal normal subgroups} is just the quotient partial action of $\theta^{\Gamma}$ by $(\mathcal{D}, N)$ where $\mathcal{D}$ is the diagonal relation on $X$, i.e., $x \sim_{\mathcal{D}} y \Leftrightarrow x =y$.

\begin{remark}
    Observe that morphism of reduced partial actions with groupoid $\Gamma$ are determined exclusively by a morphism between the respective groups. Furthermore, if such morphism of groups is an isomorphism, then by Lemma~\ref{l: isomorphism of partial actions via groupoids} we have that both reduced partial actions are isomorphic.
\end{remark}

Using the characterization of reduced partial actions with groupoid $\Gamma$ given by \eqref{eq: admisible normal subgroups} we obtain the next theorem:
\begin{theorem}
    Let $\theta$ be a partial action of $G$ on $Y$. Then, the associated groupoid of $\theta$ is $\Gamma {\rightrightarrows} X$ if, and only if, there exists $N \in \mathcal{W}_{\Gamma}$ such that $\operatorname{red} \theta$ is isomorphic to the quotient partial action $\overline{\theta^{\Gamma}}$ of $\theta^{\Gamma}$ by $(\mathcal{D}, N)$, where $\mathcal{D}$ is the diagonal relation on $X$.
\end{theorem}

\begin{proof}
    First note that $\Psi(\theta) = \Psi(\operatorname{red}(\theta))$. Thus, we can assume that $\theta$ is a reduced partial action. Finally, Remark~\ref{r: partial action restriction to X}, Proposition~\ref{p: groupoid partial actions are determined by phi} and Proposition~\ref{p: characterization of ideal normal subgroups} give us the desired conclusion.
\end{proof}

Observe that $(\mathcal{W}_\Gamma, \subseteq)$ is a partial order set with the usual order of sets. If $\{ N_i \}_{i \in I}$ is a chain in $(\mathcal{W}_\Gamma, \subseteq)$, then $N = \cup_{i \in I} N_i \in \mathcal{W}_\Gamma$. Indeed, we know that $N$ is a normal subgroup since is the union of a chain of normal subgroups. Let $ g \in N$, $g \neq 1_{G_{\Gamma}}$, then $g \in N_i$ for some $i \in I$, thus $g \neq \lceil \gamma \rceil$ for all $\gamma \in \Gamma$. Henceforth, by the Zorn's Lemma, we conclude that $\mathcal{W}_\Gamma$ has maximal elements.

The order of $\mathcal{W}_{\Gamma}$ determines an order in the set of reduced partial actions with groupoid $\Gamma$ as follows: if $\alpha$ is a partial action of $G_{\Gamma}/N$ on $X$ and $\beta$ a partial action of $G_{\Gamma}/K$ on $X$, where $N, K \in \mathcal{W}_{\Gamma}$ then we have that
\[
   N \subseteq K \Longleftrightarrow \text{ the map } G_{\Gamma}/N \to G_{\Gamma}/K \text{ is a well-defined surjective map.} 
\]
In other words, $\alpha \leq \beta$ if, and only if, there exists a surjective morphism of groups $\varphi$ such that $(1_{X}, \varphi) : \alpha \to \beta$ is a morphism of partial actions.

\begin{definition}
    Let $\alpha = (S, X, \{ D_h \}, \{ \alpha_h \})$ be a reduced partial action with groupoid $\Gamma$. We say that $\alpha$ is a maximal reduced partial action if the kernel of the respective map $\phi: G_{\Gamma} \to S$ is a maximal element of $\mathcal{W}_{\Gamma}$. 
    In particular for any other reduced partial action $\beta = (S', X, \{ D'_h \}, \{ \beta_t \})$, with respective morphism $\phi': G_{\Gamma} \to S'$, and a morphism of partial actions $(1_X, \varphi): \alpha \to \beta$ such that the following diagram commutes 
   \[\begin{tikzcd}
      {G_\Gamma} \\
      S & {S',}
      \arrow["\phi"', from=1-1, to=2-1]
      \arrow["{\phi'}", from=1-1, to=2-2]
      \arrow["\varphi"', from=2-1, to=2-2]
   \end{tikzcd}\]
   we have that $\varphi$ is a group isomorphism, and consequently $(1_X, \varphi)$ is an isomorphism of partial actions.
   \label{d: maximal reduced partial action}
\end{definition}

Note that by Proposition~\ref{p: groupoid partial actions are determined by phi} we know that reduced partial actions are determined by their respective induced group morphism. Therefore, the commutativity of the diagram in Definition~\ref{d: maximal reduced partial action} is equivalent to
   \[\begin{tikzcd}
      {\theta_\Gamma} \\
      \alpha & \beta.
      \arrow[from=1-1, to=2-1]
      \arrow[from=1-1, to=2-2]
      \arrow["{(1_X,\varphi)}"', from=2-1, to=2-2]
   \end{tikzcd}\]
This, means that if a morphism of reduced partial actions factorizes through a maximal partial action, then both partial actions are isomorphic. Heuristically, a maximal partial action with groupoid $\Gamma$ is a partial action associated groupoid $\Gamma$ with the largest possible domains, since when we make the quotient partial action of $\theta^{\Gamma}$ by a congruence determined by an element of $\mathcal{W}_{\Gamma}$, we are expanding the domains by \textit{reducing} the group that acts on $X$. By the above discussion we have the following proposition:

\begin{proposition}
    For any groupoid $\Gamma$ there exists maximal reduced partial actions with groupoid $\Gamma$. In particular, global actions with groupoid $\Gamma$ are maximal.
\end{proposition}

One can construct examples of maximal partial actions associated with a connected groupoid $\Gamma$ in which such partial action is not global actions.

\begin{example} \label{e: disk without zero}
    Let $X := \{ z \in \mathbb{C} : 0 < \| z \| \leq 1 \}$. Let $\Gamma$ be the tree groupoid with objects $X$, i.e., $\hom(x,y)$ consists only of one isomorphism. We define the partial action $\theta := \big( S^{1} \times \mathbb{R}, X, \{ D_{(z, r)} \}, \{ \theta_{(z, r)} \} \big)$  such that 
    \begin{enumerate}[(i)]
        \item $D_{(z,r)}:= X$ for all $z \in S^1$ and $r \leq 0$,
        \item $D_{(z,r)}:= \{ z \in X : \| z \| \leq 2^{-r} \}$ for all $z \in S^1$ and $r > 0$,
        \item $\theta_{(z,r)}(x) := 2^{r}zx$,
    \end{enumerate}
    where $S^1 \subseteq \mathbb{C}$ is the unit circle, and $\mathbb{R}$ is the group with the additive structure. Notice that $\theta_{(z,x)} \neq \emptyset$ for all $(z,r) \in S^1 \times \mathbb{R}$, this shows that $\theta$ is a reduced partial action with groupoid $\Gamma$. Furthermore, the respective domains of each $(z,r)$ cannot be enlarged, then $\theta_{(z,r)}$ is not a restriction of other partial bijection, thus $\theta$ is maximal.
\end{example}


\section*{Acknowledgments}

The author was supported by Funda\c c\~ao de Amparo \`a Pesquisa do Estado de S\~ao Paulo (Fapesp), process n°: 2022/12963-7.


\bibliographystyle{abbrv}
\bibliography{azu}

\end{document}